\documentclass[11pt]{article}
\usepackage[a4paper,left=20mm,top=20mm,right=20mm,bottom=25mm]{geometry}
\usepackage[mathscr]{euscript}
\usepackage{amsmath,amsfonts,amsthm,amssymb}
\usepackage{mathtools}
\usepackage{graphicx}
\usepackage{enumerate}
\usepackage{authblk}

\newtheorem{theorem}{Theorem}[section]
\newtheorem{definition}{Definition}[section]
\newtheorem{lemma}[theorem]{Lemma}
\newtheorem{example}[theorem]{Example}
\newtheorem{proposition}[theorem]{Proposition}
\newtheorem{fact}[theorem]{Observation}
\newtheorem{corollary}[theorem]{Corollary}

\newcommand{\AX}[1]{\begingroup\bfseries\upshape #1\endgroup}
\newcommand{\overlaps}{\between}
\begin{document}
\title{Transit Functions and Pyramid-Like Binary Clustering Systems}

\author[1]{Manoj Changat}
\author[1]{Ameera Vaheeda Shanavas}
\author[2]{Peter F.\ Stadler}

\affil[1]{Department of Futures Studies, University of Kerala,
  Trivandrum, IN 695 581, India\\
  mchangat@gmail.com, ameerasv@gmail.com}

\affil[2]{Bioinformatics Group, Department of Computer Science \&
  Interdisciplinary Center for Bioinformatics, Universit{\"a}t Leipzig,
  H{\"a}rtelstra{\ss}e 16-18, D-04107 Leipzig, Germany; German Centre for
  Integrative Biodiversity Research (iDiv) Halle-Jena-Leipzig, Competence
  Center for Scalable Data Services and Solutions Dresden-Leipzig, Leipzig
  Research Center for Civilization Diseases, and Centre for Biotechnology
  and Biomedicine at Leipzig University at Universit{\"a}t Leipzig\\
  Max Planck Institute for Mathematics in the Sciences, Inselstra{\ss}e 22,
  D-04103 Leipzig, Germany \\
  Institute for Theoretical Chemistry, University of Vienna,
  W{\"a}hringerstrasse 17, A-1090 Wien, Austria\\
  Facultad de Ciencias, Universidad National de Colombia
  Bogot{\'a}, Colombia\\
  Santa Fe Insitute, 1399 Hyde Park Rd., Santa Fe NM 87501, USA\\
  studla@bioinf.uni-leipzig.de}
\date{\ }

\setcounter{Maxaffil}{0}
\renewcommand\Affilfont{\scriptsize}

\maketitle

\begin{abstract}
  Binary clustering systems are closely related to monotone transit
  functions. An interesting class are pyramidal transit functions defined
  by the fact that their transit sets form an interval hypergraph. We
  investigate here properties of transit function $R$, such as
  union-closure, that are sufficient to ensure that $R$ is at least weakly
  pyramidal. Necessary conditions for pyramidal transit functions are
  derived from the five forbidden configurations in Tucker's
  characterization of interval hypergraphs. The first corresponds to
  $\beta$-acyclicity, also known as total balancedness, for which we obtain
  three alternative characterizations. For monotonous transit functions,
  the last forbidden configuration becomes redundant, leaving us with
  characterization of pyramidal transit functions in terms of four
  additional conditions.
\end{abstract}
  

\bigskip\bigskip

\section{Introduction} 

Transit functions have been introduced as a unifying approach for results
and ideas on intervals, convexities, and betweenness in graphs and posets
\cite{Mulder:08}. Initially, they were introduced to capture an abstract
notion of ``betweenness'', i.e., an element $x$ is considered to be
``between'' $u$ and $v$, if $x\in R(u,v)$. Monotone transit functions,
satisfying $R(x,y)\subseteq R(u,v)$ for all $x,y\in R(u,v)$, play an
important role in the theory of convexities on graphs and other set systems
\cite{vandeVel:93}. They also have an alternative interpretation as
\emph{binary clustering systems} \cite{Barthelemy:08}. Here every cluster
is ``spanned'' by a pair of points in the following sense: if $C$ is a
cluster, then there are two points $x$ and $y$, such that $C$ is the unique
inclusion-minimal cluster containing $x$ and $y$. Binary clustering systems
include not only hierarchies but also important clustering models with
overlaps, such as paired hierarchies \cite{Bertrand:08}, pyramids
\cite{Diday:86,Bertrand:13}, and weak hierarchies \cite{Bandelt:89}.

This connection between transit functions and clustering systems suggests
investigating ``natural'' properties of transit function as motivation for
properties of clustering systems. Hierarchies and paired hierarchies are
frequently too restrictive. On the other hand, weak hierarchies may already
be too general e.g.\ when considering clustering systems of phylogenetic
networks \cite{Hellmuth:22q}. Pyramids, i.e., clustering systems whose
elements can be seen as intervals, form a possible middle ground. These set
systems are equivalent to interval hypergraphs
\cite{Tucker:72,Trotter:76,Nebesky:83,Duchet:84}. A potential shortcoming
of pyramids is that the associated total order of the points may not always
have a simple interpretation and that these set systems do not have a
compact characterization without (implicit) reference to an ordering of the
points.  It is of interest, therefore, to consider binary clustering
systems that are either mild generalizations or restrictions of
pyramids. Naturally, the existence of total order in an interval hypergraph
closely liked to notions of acyclicity in hypergraphs, which were
originally studied in the context of database schemes
\cite{Fagin:83,Ausiello:85}. Here, we connect these concepts with
corresponding properties of transit functions and binary clustering
systems.

This contribution is organized as follows. In Section~\ref{sect:notation},
we provide the technical background and connect properties of transit
functions with the literature on cycle types and acyclicity in hypergraphs.
It is proved in \cite{Changat:21w} that the union closed binary clustering
systems are pyramidal. Section~\ref{sect:uc} considers union-closed set
systems and derives an alternative characterization of their canonical
transit functions, thereby solving an open problem of \cite{Changat:21w}.
In section~\ref{sect:wp}, we consider the weak pyramidality property and a
relaxed variant termed \AX{(i)} that is closely related to axioms studied
in earlier work \cite{Changat:21w}. We prove that weakly pyramidal set
systems and weak hierarchies satisfying \AX{(I)} are equivalent and that
the corresponding axiom \AX{(i)} for transit functions is in between
\AX{(wp)} and \AX{(o')}.  In section~\ref{sect:l1l2}, we consider
sufficient conditions for weakly pyramidal transit functions. Here we
consider two clustering systems between paired hierarchies and weakly
pyramidal set systems that are described by the axioms \AX{(L1)} and
\AX{(N3O)}, and study their identifying transit functions. We introduce
axiom \AX{(l2)} and prove that a monotone transit function satisfying both
\AX{(l1)} and \AX{(l2)} is pyramidal.

Then we turn to the necessary conditions for pyramidal transit functions:
Section~\ref{sect:tb} is concerned with more general, so-called totally
balanced transit functions, whose set systems are $\beta$-acyclic
hypergraphs \cite{Anstee:83,Lehel:85} and thus do not contain the first
forbidden configuration in Tucker's \cite{Tucker:72} characterization of
interval hypergraphs. We derive three alternative characterizations for
this property for totally balanced monotone transit functions. We then
proceed in Sections~\ref{sect:23} and \ref{sect:45} to analyze the
forbidden configurations in Tucker's \cite{Tucker:72} characterization and
finally arrive at a characterization of pyramidal transit functions. We
close this contribution in section \ref{discsn} with a summary of the
relationships among the properties considered here and some open questions.

\section{Notation and Preliminaries} 
\label{sect:notation}
Throughout, $V$ is a non-empty finite set and $\mathscr{C}\subseteq 2^V$ is
a set of subsets of $V$ that does not contain the empty set.  Two sets
$A,B\subseteq V$ overlap, in symbols $A\overlaps B$, if $A\cap B$,
$A\setminus B$ and $B\setminus A$ are non-empty.
  
\paragraph{Transit Functions and $\mathscr{T}$-Systems} 
Formally, a \emph{transit function} \cite{Mulder:08} on a non-empty set $V$
is a function $R: V\times V\to 2^V$ satisfying the three axioms
\begin{description}\setlength{\itemsep}{0pt}\setlength{\parskip}{0pt}%
\item[\AX{(t1)}] $u\in R(u,v)$ for all $u,v\in V$.
\item[\AX{(t2)}] $R(u,v)=R(v,u)$ for all $u,v\in V$.
\item[\AX{(t3)}] $R(u,u)=\{u\}$ for all $u\in V$.
\end{description}
The transit functions appearing in the context of clustering systems
\cite{Barthelemy:08} in addition satisfy the \emph{monotonicity axiom}
\begin{description}\setlength{\itemsep}{0pt}\setlength{\parskip}{0pt}%
\item[\AX{(m)}] $p,q\in R(u,v)$ implies $R(p,q)\subseteq R(u,v)$. 
\end{description}
Transit functions are sometimes called ``Boolean dissimilarities''. The set
systems corresponding to monotone transit functions are slightly more
general than binary clustering systems. A characterization was given in
\cite{Changat:19a}:
\begin{definition}
  A \emph{$\mathscr{T}$-system} is a system of non-empty sets
  $\mathscr{C}\subset 2^V$ satisfying the three axioms
  \begin{description}\setlength{\itemsep}{0pt}\setlength{\parskip}{0pt}%
  \item[\AX{(KS)}] $\{x\}\in\mathscr{C}$ for all $x\in V$.
  \item[\AX{(KR)}] For every $C\in\mathscr{C}$ there are points $p,q\in C$
    such that $p,q\in C'$ implies $C\subseteq C'$ for all
    $C'\in \mathscr{C}$.
  \item[\AX{(KC)}] For any two $p,q\in V$ holds
    $\displaystyle \bigcap\{C\in\mathscr{C}|p,q\in C\} \in \mathscr{C}$.
  \end{description}
\end{definition}
We say that a set system $\mathscr{C}$ \emph{is identified} by the transit
function $R$ if $\mathscr{C}=\{ R(x,y)| x,y\in V\}$.
\begin{proposition} \cite{Changat:19a}
  There is a bijection between monotone transit functions $R:V\times V\to
  2^V$ and $\mathscr{T}$-systems $\mathscr{C}\subseteq 2^V$ mediate by  
  \begin{equation}
    \begin{split}
      R_{\mathscr{C}}(x,y) &\coloneqq \bigcap\{ C\in\mathscr{C}| x,y\in C\}\\
      \mathscr{C}_R        &\coloneqq \{ R(x,y)| x,y\in V\}
    \end{split}
  \end{equation}
\end{proposition}
That is, a set system $\mathscr{C}$ is identified by a transit function $R$
if and only if $\mathscr{C}$ is a $\mathscr{T}$-system. In this case,
$R=R_{\mathscr{C}}$ and we call $R_{\mathscr{C}}$ the \emph{canonical
transit function} of the set system $\mathscr{C}$, and $\mathscr{C}_R$ the
collection of \emph{transit sets} of $R$. A set system is binary in the
sense of \cite{Barthelemy:08} if and only if it satisfies \AX{(KC)} and
\AX{(KR)}.  Axiom \AX{(KS)} corresponds to \AX{(t3)}, i.e., the fact that
all singletons are transit sets.  Moreover, a $\mathscr{T}$-system is a
\emph{binary clustering system} if it satisfies
\begin{description}\setlength{\itemsep}{0pt}\setlength{\parskip}{0pt}%
  \item[\AX{(K1)}] $V \in \mathscr{C}$
\end{description} 
As shown in \cite{Barthelemy:08,Changat:19a}, binary clustering systems are
identified by monotone transit function satisfying the additional condition
\begin{description} 
\item[\AX{(a')}] There exist $u,v\in V$ such that $R(u,v)=V$.
\end{description}

\paragraph{Conformal Hypergraphs} The system of transit sets
$\mathscr{C}_R$ has a natural interpretation as the edge set of the
hypergraph $(V,\mathscr{C}_R)$. The \emph{primal graph} or
\emph{two-section} $H_{[2]}$ of $(V,\mathscr{C}_R)$ is the graph with
vertex set $V$ and an edge $\{x,y\}\in E(H_{[2]})$ whenever there is
$C\in\mathscr{C}_R$ with $\{x,y\}\subseteq C$. Axiom \AX{(t1)} and the fact
that $R$ is defined on $V\times V$ implies that the primal graph $H_{[2]}$
of $(V,\mathscr{C}_R)$ is the complete graph on $V$.  A hypergraph is said
to be \emph{conformal} \cite{Tarjan:84} if every clique of its two-section
is a hyperedge or, equivalently, covered by a hyperedge. Therefore we have
\begin{fact}
  Let $R$ be a monotone transit function. Then $(V,\mathscr{C}_R)$ is a
  conformal hypergraph if and only if $R$ satisfies \AX{(a')}.
\end{fact}
Since complete graphs are chordal, and a hypergraph is $\alpha$-acylic if
and only if it is conformal and its two-section is chordal \cite{Beeri:83},
we obtain immediately:
\begin{corollary}
  Let $R$ be a monotone transit function. Then $(V,\mathscr{C}_R)$ is an
  $\alpha$-acyclic hypergraph if and only if $R$ satisfies \AX{(a')}.
\end{corollary}

\paragraph{Convexities}
A set system $(V,\mathscr{C})$ is \emph{closed (under non-empty
intersection)} if it satisfies
\begin{description}\setlength{\itemsep}{0pt}\setlength{\parskip}{0pt}%
\item[\AX{(K2)}] If $A,B\in\mathscr{C}$ and $A\cap B\ne\emptyset$ then
  $A\cap B\in\mathscr{C}$. 
\end{description}
$(V,\mathscr{C})$ is a \emph{convexity} if it satisfies \AX{(K1)} and
\AX{(K2)}. We remark that convexities are usually defined to include the
empty set and thus without the restriction of \AX{(K2)} to $A\cap
B\ne\emptyset$. Here, we insist on $\emptyset\notin\mathscr{C}$ in order
for \emph{grounded convexities}, i.e., those satisfying \AX{(KS)}, to be
the same as closed clustering systems. Note that \AX{(K2)} implies
\AX{(KC)}, while the converse is not true. In the language of monotone
transit functions, \AX{(K2)} is equivalent to
\begin{description}\setlength{\itemsep}{0pt}\setlength{\parskip}{0pt}%
\item[\AX{(k)}] For all $u,v,x,y\in V$ with $R(u,v)\cap
  R(x,y)\ne\emptyset$, there exist $p,q\in V$ such that $R(u,v)\cap R(x,y) =
  R(p,q)$.
\end{description}
Axiom \AX{(k)} was introduced as \AX{(m')} in
\cite{Changat:18a,Changat:19a} and renamed \AX{(k)} in later work to
emphasize that it is unrelated to the monotonicity axiom \AX{(m)}.

\paragraph{Weak Hierarchies}
A clustering system $\mathscr{C}$ is a \emph{weak hierarchy}
\cite{Bandelt:89} if for any three sets $A,B,C\in\mathscr{C}$ holds
\begin{equation}
  A \cap B \cap C \in \{ A\cap B, A\cap C, B\cap C\}
\end{equation}
The family of transit sets of the canonical transit function of a weak
hierarchy is a convexity, and hence closed \cite{Changat:19a}.
The corresponding canonical transit function satisfies
\begin{description}\setlength{\itemsep}{0pt}\setlength{\parskip}{0pt}%
\item[\AX(w)] For all $x,y,z\in V$ holds $z\in R(x,y)$ or $y\in R(x,z)$ or
  $x\in R(y,z)$.
\end{description}
In \cite{Changat:21w}, many properties equivalent to \AX{(w)} are
listed. Lemma~1 of \cite{Bandelt:89} implies that weak hierarchies always
satisfy \AX{(KR)}. Moreover, for monotone transit functions, \AX{(w)}
implies both \AX{(k)} \cite[Lemma 4.2]{Changat:19a} and \AX{(a')}. The
corresponding clustering system thus satisfies \AX{(K2)}.

\paragraph{Pyramids}
Clustering structures with an additional, linear ordering are known as
pyramids \cite{Diday:86,Bertrand:13} or interval hypergraphs
\cite{Duchet:78}.
\begin{definition} 
  A clustering system $(V,\mathscr{C})$ is \emph{pre-pyramidal} if there
  exists a total order $<$ on $V$ such that for every $C\in\mathscr{C}$ and
  all $x,y\in C$ we have $x<u<y$ implies $u\in C$. That is, all clusters
  $C\in\mathscr{C}$ are intervals w.r.t.\ $<$.
\end{definition}
As shown in \cite{Tucker:72,Trotter:76,Duchet:84}, pre-pyramidal set
systems are characterized by the following (infinite) series of forbidden
induced sub-hypergraphs:
\begin{equation}
  \label{eq:inthypg}
\includegraphics[width=0.9\textwidth]{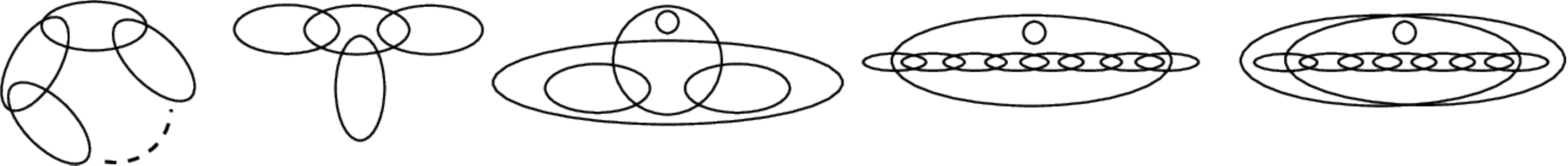}
\end{equation}
Note that the forbidden configurations are to be interpreted as restrictions
of a hypergraph to a subset of its vertices. Sets shown as disjoint in the
diagrams therefore may intersect in additional vertices.  Furthermore, the
number of ``small'' overlapping sets located completely inside the large
ellipse in the fourth case and those located completely inside the
intersection of the large ellipses in the fifth case may be any $k\ge 0$.
It is well known that pre-pyramids are a proper subclass of weak
hierarchies, see e.g.\ \cite{Bertrand:17}.\\

Pre-pyramidal set systems are also known as \emph{interval hypergraphs}
\cite{Duchet:78}.  For any three distinct points $x,y,z\in V$, we say that
$y$ lies between $x$ and $z$ if every hyperpath connecting $x$ and $z$ has
an edge containing $y$. A hypergraph $(V,\mathscr{C})$ is an interval
hypergraph if and only if every set $\{x,y,z\}\subseteq V$ of three
distinct vertices contains a vertex that lies between the other two
\cite{Duchet:78}. A clustering system $(V,\mathscr{C})$ is \emph{pyramidal}
if it is pre-pyramidal and closed. Since pre-pyramidal clustering systems
are a special case of weak hierarchies \cite{Bertrand:13}, we observe that
for every monotone transit function with a pre-pyramidal system of transit
sets, $\mathscr{C}_R$ is closed and thus pyramidal.
\begin{definition}
  A monotone transit function $R$ is pyramidal if its set of transit sets
  $\mathscr{C}_R$ form a pyramidal clustering system.
\end{definition}

In other words, $R$ is pyradmidal if there is a total order $<$ on $V$
  such that $R(u,v)$ is an interval for all $u,v\in V$. It is worth noting
  that, despite Duchet's betweenness-based characterization of interval
  hypergraphs \cite{Duchet:78}, neither of the two classical betweenness
  axiom \cite{Mulder:08}
  \begin{description}
  \item[\AX{(b3)}] If $x\in R(u,v)$ and $y\in R(u,x)$, then $x\in R(y,v)$
  \item[\AX{(b4)}] If $x\in R(u,v)$, then $R(u,x)\cap  R(x,v) = \{x\}$
  \end{description}
  need to be satisfied.  The transit function $R$ on $V=\{u,v,x,y\}$
  comprising the singletons, $R(v,y)=\{v,y\}$ and $R(p,q)=V$ otherwise, is
  monotone and violates \AX{(b3)}. However, it is pyramidal with order
  $x<u<y<v$. The ``indiscrete transit function'', whose transit sets are
  only the singletons and $V$, is monotone, pyramidal, and violates
  \AX{(b4)}. In fact, the pyramidal transit function satisfying \AX{(b4)}
  is uniquely defined (up to isomorphism).
  
  Given a total order $<$ on $V$, we write $[u,v]\coloneqq\{w\in V| u\le
  w\le v\}$ for the intervals w.r.t.\ the order $<$. A monotone pyramidal
  transit function $R$ by construction satisfies $[u,v]\subseteq R(u,v)$
  for all $u,v\in V$.
  \begin{proposition}
    Let $R$ be a monotone, pyramidal transit function satisfying \AX{(b3)}
    or \AX{(b4)}.  Then $R(u,v)=[u,v]$ for all $u,v\in V$.
  \end{proposition}
  \begin{proof}
    Let $R$ be monotone and pyramidal. In particular, $R$ satisfies
    \AX{(a')}, i.e., there is $p<q$ such that $R(p,q)=[p,q]=V$.  First
    suppose $R$ satisfies \AX{(b3)}. If $x\ne q$, then $q\notin R(p,x)$
    since otherwise \AX{(b3)} implies $x\in R(q,q)=\{q\}$, contradicting
    \AX{(t3)}. If $x$ is the predecessor of $q$ w.r.t. $<$, then
    $R(p,x)=[p,x]$. Repeating the argument for the predecessor $x'\in
    R(p,x)=[p,x]$ of $x$ yields $R(p,x')=[p,x']$, and eventually
    $R(p,u)=[p,u]$ for all $u\in V$. We argue analogously that for
    $p'\in[p,x]$ with $p'\ne p$ we have $p\notin
    R(p',x)\subseteq[p,x]$. For the successor $w$ of $p$, therefore, we we
    obtain $R(w,x)=[w,x]$. Again, repeating the argument for $w'\in[w,x]$
    we eventually obtain $R(u,x)=[u,x]$ for all $u\le x$.

    Now suppose $R$ satisfies \AX{(b4)}. Assume that there is $x\in V$ such
    that $w\in R(p,x)\setminus[p,x]$ and thus $w\in R(p,x)\cap [x,q]
    \subseteq R(p,x)\cap R(x,q)\ne \{x\}$, contradicting
    \AX{(b4)}. Thus $R(p,x)=[p,x]$ for all $x\in V$. Similarly, we obtain
    $R(x,q)=[x,q]$ for all $x\in V$. Now consider $x'\in R(p,x)$ and
    suppose $R(x',x)\ne[x',x]$, i.e., there is $w\in
    R(x',x)\setminus[x',x]$. By monotonicity, $w\in R(p,x)=[p,x]$ and thus
    $w\in R(p,x')\cap R(x',x)=[p,x']\cap R(x',x)\ne\{x'\}$, again
    contradicting \AX{(b4)}.  Thus $R(x',x)=[x',x]$. An anlogous argument
    shows $R(x,x')=[x,x']$ for all $x'\in R(x,q)$.
  \end{proof}

\paragraph{Totally balanced clustering systems}  
The first forbidden configuration in Eq.(\ref{eq:inthypg}) amounts to the
absence of a so-called weak $\beta$-cycle \cite{Fagin:83,Cheung:91}, which
is defined as a sequence of $n\ge 3$ sets $C_1,\dots,C_n$ and vertices
$x_1,\dots , x_n$ such that, for all $i$, $x_i\in C_{i}\cap C_{i+1}$ and
$x_i\notin C_k$ for any $k\notin\{i,i+1\}$ (where indices are taken modulo
$n$). Hypergraphs without a weak $\beta$-cycle are known as
\emph{$\beta$-acyclic} or \emph{totally balanced}, see
e.g.\ \cite{Anstee:83,Lehel:85}.  Since a $\beta$-cycle of length $3$
amounts to $C_1,C_2,C_3\in\mathscr{C}$ and points $x_1,x_2,x_3$ such that
$x_i\in C_i\cap C_{i+1}$ but $x_i\notin C_{i+2}$, we note that
$\mathscr{C}$ is a weak hierarchy if and only if it does not contain a
$\beta$-cycle of length $3$ \cite{Bandelt:89}. From the point of view of
clustering systems, totally balanced ($\beta$-acyclic) hypergraphs have
been studied in \cite{Brucker:20}.
  
\paragraph{Paired Hierarchies} A set system $\mathscr{C}$ is a paired
hierarchy if every cluster $C\in\mathscr{C}$ overlaps at most one other
cluster $C'\in\mathscr{C}$ \cite{Bertrand:08}. A characterization of paired
hierarchies in terms of their transit functions can be found in
\cite{Bertrand:17,Changat:19a}.

\paragraph{Hierarchies} A set system $\mathscr{C}$ on $V$ is a hierarchy
if $V\in\mathscr{C}$, all singletons belong to $\mathscr{C}$ and $A\cap
B\in\{A,B,\emptyset\}$ for all $A,B\in\mathscr{C}$. Several alternative
characterizations of monotone transit functions whose transit sets form a
hierarchy are discussed in \cite{Changat:18a}, see also \cite{Bertrand:17}.

\paragraph{Acyclicity in Hypergraphs}
There is extensive literature concerned with notions of acyclicity in
hypergraphs \cite{Fagin:83,Ausiello:85,Duris:12,Brault:17}. A stronger
condition than $\beta$-acyclicy is $\gamma$-acyclicity, which can be
phrased as follows: A hypergraph is $\gamma$-acyclic if it contains neither
a \emph{pure cycle} nor a so-called $\gamma$-triangle. A cycle is pure if
it is a $\beta$-cycle such that $C_i\cap C_j=\emptyset$ for $j\ne
i,i-1,i+1$. A $\gamma$-triangle consists of three pairwisely intersecting
sets $C_1,C_2,C_3\in\mathscr{C}$ such that there exist $u,v\in C_3$ with $u\in
C_1\setminus C_2$ and $v\in C_2\setminus C_1$ \cite{Fagin:83,Ausiello:85}.

If $\mathscr{C}$ is not a weak hierarchy, then there are three pairwisely
overlapping sets $C_1,C_2,C_3\in\mathscr{C}$ such that $C_1\cap C_2\cap
C_3\notin\{C_1\cap C_2,C_1\cap C_3,C_2\cap C_3\}$.  From $C_1\cap C_2\cap
C_3\subsetneq C_1\cap C_3$, we note that there is $u\in C_1\cap C_3$ with
$u\notin C_2$, and $C_1\cap C_2\cap C_3\subsetneq C_2\cap C_3$ implies
$v\in C_2\cap C_3$ with $v\notin C_1$, i.e., $\{C_1,C_2,C_3\}$ forms a
$\gamma$-triangle. Thus $\gamma$-acyclicity implies that $\mathscr{C}$ is a
weak hierarchy.  Conversely, if $\mathscr{C}$ satisfies \AX{(W)}, then
every $\gamma$-triangle is of the form $C_1\overlaps C_2$ and $\emptyset\ne
C_1\cap C_2\subseteq C_3$. Furthermore, if $R$ is a transit function
satisfying \AX{(w)} then we have $V\in\mathscr{C}_R$ \cite{Changat:19a},
and thus any pair of overlapping clusters $C_1\overlaps C_2$ together with
$V$ forms a $\gamma$-triangle. Thus a $\gamma$-acyclic weak hierarchy
contains no overlapping clusters and thus are hierarchies. Conversely,
hierarchies are $\gamma$-acyclic since every pure cycle consists of a
sequence of consecutive overlapping clusters, and every $\gamma$-triangle
contains two overlapping sets. Therefore, we have
\begin{proposition}
  Let $R$ be a monotone transit function. Then $\mathscr{C}_R$ is
  $\gamma$-acyclic if and only if $\mathscr{C}_R$ is a hierarchy.
\end{proposition}

\section{Union-Closed Set Systems}
\label{sect:uc}

A set system $(V,\mathscr{C})$ is \emph{union-closed} if it contains all
singletons and satisfies
\begin{description}\setlength{\itemsep}{0pt}\setlength{\parskip}{0pt}%
\item[(UC)] If $C',C''\in\mathscr{C}$ and $C'\cap C''\ne\emptyset$ then
  $C'\cup C''\in\mathscr{C}$.
\end{description}

In \cite{Changat:21w} we considered the following two properties:
\begin{description}\setlength{\itemsep}{0pt}\setlength{\parskip}{0pt}%
\item[\AX{(uc)}] If $R(x,y)\cap R(u,v)\ne\emptyset$ then there exist
  $p,q\in R(x,y)\cup R(u,v)$ such that
  $R(x,y)\cup R(u,v)=R(p,q)$.
\item[\AX{(u)}] If $z\in R(u,v)$ then $R(u,v)=R(u,z)\cup R(z,v)$.
\end{description}
A monotone transit function satisfies \AX{(uc)} if and only if the
corresponding $\mathscr{T}$-system satisfies \AX{(UC)}. Axiom \AX{(u)}
appeared as a property of cut-vertex transit functions of hypergraphs in
\cite{Changat:21a} and as property \AX{(h'')} in \cite{Changat:18a} in the
context of characterizing hierarchical clustering systems. It was studied
further in \cite{Changat:21w}, where it was left as an open question
whether weak hierarchies with canonical transit functions that satisfy
\AX{(u)} are union-closed. The following result gives an affirmative
answer:

\begin{theorem}
  \label{thm:uc<=>u,w}
  If $R$ is a monotone transit function. Then $R$ satisfies \AX{(uc)} if
  and only if it satisfies \AX{(u)} and \AX{(w)}.
\end{theorem}
\begin{proof}
  Suppose $R$ satisfies \AX{(uc)}. Then \cite[Thm.3]{Changat:21w} shows
  that $R$ satisfies \AX{(w)}, and \cite[Lemma~8]{Changat:21w} establishes
  that $R$ satisfies \AX{(u)}.

  For the converse, assume that $R$ satisfies \AX{(u)} and \AX{(w)} and
  suppose that $R(x,y)\cap R(u,v)\neq \emptyset$ and there exists $p\in
  R(x,y)\setminus R(u,v)$ and $q\in R(u,v)\setminus R(x,y)$ with
  $R(p,q)\neq R(x,y)\cup R(u,v)$. Let $a\in R(x,y)\cap R(u,v)$, the
  \AX{(w)} (more precisely the equivalent condition \AX{(w$_3$)} in
  \cite[Thm.~1]{Changat:21w}) implies $a\in R(p,q)$ and \AX{(u)} yields
  $R(p,q)=R(p,a)\cup R(a,q)\subset R(x,y)\cup R(u,v)$. By assumption, there
  exists $d_1\in R(x,y)\cup R(u,v)$ such that $d_1\notin R(p,q)$.

  Suppose $d_1\in R(x,y)\setminus R(u,v)$. Invoking \AX{(w$_3$}) again, we
  obtain $a\in R(d_1,q)$ and thus \AX{(u)} implies $R(d_1,q)=R(d_1,a)\cup
  R(a,q)$. Moreover, $R(p,q)=R(p,a)\cup R(a,q)$. From $p,d_1\in R(x,y)$ we
  have $R(p,d_1)\subseteq R(x,y)$ by monotonicity, and therefore $q\notin
  R(p,d_1)$. If $p\in R(d_1,q)$, then $p\in R(d_1,a)\cup R(a,q)$.  The case
  is impossible since $p\in R(a,q)$ implies $p\in R(u,v)$ contradicting our
  assumptions; thus $p\in R(a,d_1)$. By contraposition, $p\notin R(a,d_1)$
  implies $p\notin R(d_1,q)$; in this case $p$, $q$, and $d_1$ violate
  \AX{(w)}. Therefore $p\in R(d_1,a)$ and thus $R(p,a)\subset R(d_1,a)$
  by monotonicity. This implies 
  $R(p,q)=R(p,a)\cup R(a,q)\subset R(d_1,a)\cup R(a,q)=R(d_1,q)$.

  If $R(d_1,q)= R(x,y)\cup R(u,v)$ we are done. Otherwise, there exists $d_2\in
  R(x,y)\cup R(u,v)$ with $d_2\notin R(d_1,q)$. Suppose $d_2\in
  R(u,v)\setminus R(x,y)$. Thus $d_2,q\in R(u,v)$ and monotonicity implies
  $R(d_2,q)\subseteq R(u,v)$, which in turn implies $d_1\notin
  R(d_2,q)$. Now if $q\in R(d_1,d_2)$ then $q\in R(d_1,a)\cup
  R(a,d_2)$. Since $a,d_1\in R(u,v)$ we have $q\notin R(d_1,a)$ and thus
  $q\in R(a,d_2)$. That is $q\notin R(a,d_2)$ implies $q\notin R(d_1,d_2)$
  and thus $d_1$, $d_2$, and $q$ violate \AX{(w)}. Therefore, we must have
  $q\in R(a,d_2)$ and hence $R(a,q)\subset R(d_2,a)$. This implies
  $R(d_1,q)=R(d_1,a)\cup R(a,q)\subset R(d_1,a)\cup
  R(a,d_2)=R(d_1,d_2)$. If $R(d_1,d_2)= R(x,y)\cup R(u,v)$, we are done.

  Otherwise, there exists $d_3\in R(x,y)\cup R(u,v)$ such that $d_3\notin
  R(d_1,d_2)$. Using the same arguments, there exists a larger $R(d_3,d_2)$
  or $R(d_3,d_2)$ properly containing $R(d_1,d_2)$ and contained in
  $R(x,y)\cup R(u,v)$. Since $R(x,y)\cup R(u,v)$ is finite, this iteration
  reaches a pair of points $s\in R(x,y)\setminus R(u,v)$, and $t\in
  R(u,v)\setminus R(x,y)$ such that $R(s,t)=R(x,y)\cup R(u,v)$.  Thus $R$
  satisfies \AX{(uc)}.
\end{proof}
\begin{proposition}\cite{Changat:21w}
  \label{prop:uc=>py}
  A monotone transit function that satisfies \AX{(uc)} is pyramidal.
\end{proposition}
Fig.~4 in \cite{Changat:21w} shows that the converse is not true.
Thm.~\ref{thm:uc<=>u,w} and Prop.~\ref{prop:uc=>py} together answer
affirmatively the questions in \cite{Changat:21w} whether \AX{(u)} and
\AX{(w)} together are sufficient to imply that a monotone transit function
$R$ is pyramidal.

Paired hierarchies and union-closed binary clustering systems are proper
sub-classes of pyramidal clustering systems. There are, however, binary
clustering systems that are union-closed but not paired hierarchies and
clustering systems that are paired hierarchies but not union-closed:
\begin{example}
  \label{ex:ph-not-uc}
  Consider the monotone transit function $R$ on $V=\{a,b,c,d\}$ defined by
  $R(a,b)=\{a,b\}, R(b,c)=\{b,c\}$ and other sets are singletons and
  $V$. Here $\mathscr{C}_R$ is a paired hierarchy but not union-closed.
\end{example}

\section{Weakly Pyramidal Transit Functions}
\label{sect:wp} 

Ref.~\cite{Nebesky:83} characterizes pre-pyramidal set systems with $\{A,
B, C\}$ as those that are weak hierarchies and satisfy
\begin{description}\setlength{\itemsep}{0pt}\setlength{\parskip}{0pt}%
\item[\AX{(WP)}] If $A$, $B$, $C$ have pairwise non-empty intersections,
  then one set is contained in the union of the two others.
\end{description}
For larger set systems $\mathscr{C}$ the condition is still necessary, but
no longer sufficient. Thus the term \emph{weak pre-pyramids} has been
suggested for weak hierarchies that satisfy \AX{(WP)}. The axiom can be
translated trivially to the language of transit functions:
\begin{description}\setlength{\itemsep}{0pt}\setlength{\parskip}{0pt}%
\item[\AX{(wp)}] If $R(u,v)\cap R(x,y)\ne\emptyset$,
  $R(u,v)\cap R(p,q)\ne\emptyset$ and $R(x,y)\cap R(p,q)\ne\emptyset$ then
  $R(p,q)\subseteq R(u,v)\cup R(x,y)$ or
  $R(u,v)\subseteq R(p,q)\cup R(x,y)$ or
  $R(x,y)\subseteq R(p,q)\cup R(u,v)$.
\end{description}

The third forbidden configuration in Eq.(\ref{eq:inthypg}) suggests
considering set systems satisfying the following property:
\begin{description}\setlength{\itemsep}{0pt}\setlength{\parskip}{0pt}%
\item[\AX{(I)}] Let $A,B,C\in\mathscr{C}$, $\emptyset\neq A\cap B\subseteq
  C$, and $C\setminus(A\cup B)\ne\emptyset$. Then $A\subseteq C$ or
  $B\subseteq C$.
\end{description}
Let us now consider the following related property for transit functions:
\begin{description}\setlength{\itemsep}{0pt}\setlength{\parskip}{0pt}%
\item[\AX{(i)}] If $\emptyset\ne R(x,y)\cap R(u,v)\subseteq R(p,q)$ and
  $R(p,q)\setminus \left(R(x,y)\cup R(u,v)\right)\neq \emptyset$, then
  $R(x,y)\subseteq R(p,q)$ or $R(u,v)\subseteq R(p,q)$.
\end{description}
\begin{fact}
  Let $R$ be a monotone transit function, and $\mathscr{C}_R$ the
  corresponding set of transit sets. Then $\mathscr{C}_R$
  satisfies \AX{(I)} if and only if $R$ satisfies \AX{(i)}.
\end{fact}

\begin{lemma}
	\label{lem:WP->I}
  Suppose $\mathscr{C}$ satisfies \AX{(WP)}. Then $\mathscr{C}$ also
  satisfies \AX{(I)}.
\end{lemma}
\begin{proof}
  Let $A,B,C\in \mathscr{C}$ such that $\emptyset\ne A\cap B\subseteq C$
  and $C\setminus(A\cup B)\ne\emptyset$.  Then \AX{(WP)} implies
  $A\subseteq B\cup C$ or $B\subseteq A\cup C$.  In the first case, we
  obtain $A=(A\cap B)\cup(A\cap C)\subseteq C\cup(A\cap C)=C$. Similarly,
  in the second case, we obtain $B\subseteq C$, and thus \AX{(I)} holds.
\end{proof}

\begin{lemma}
	\label{lem:wp->i}
  Let $R$ be a monotone transit function satisfying \AX{(wp)}, then $R$
  satisfies \AX{(i)}.
\end{lemma}
\begin{proof}
  Consider $\emptyset\neq R(x,y)\cap R(u,v)\subseteq R(p,q)$ where
  $R(p,q)\setminus (R(x,y)\cup R(u,v))\neq \emptyset$ and
  assume, for contradiction, that neither $R(x,y)\subseteq R(p,q)$ nor
  $R(u,v)\subseteq R(p,q)$ is true.  Then there exists $w_1\in
  R(x,y)\setminus R(p,q)$, and $w_2\in R(u,v)\setminus R(p,q)$. Since
  $R(x,y)\cap R(u,v)\subseteq R(p,q)$, we
  obtain $w_1,w_2\notin R(x,y)\cap R(u,v)$. Moreover, $R(p,q)\setminus
  (R(x,y)\cup R(u,v))\neq \emptyset $. Hence none of the sets
  $R(x,y)$, $R(u,v)$, and $R(p,q)$ is contained in the union of the other
  two, contradicting \AX{(wp)}.
\end{proof}
The transit function $R$ and sets system $\mathscr{C}_R$ in
Example~\ref{i*-wp} below shows that the converses of both
Lemma~\ref{lem:WP->I} and \ref{lem:wp->i} do not hold.
\begin{example}\label{i*-wp}
  Let $R$ on $V=\{a,b,c,d,e,f\}$ be defined by
  $R(a,b)=R(a,c)=R(b,c)=\{a,b,c\}$, $R(a,d)=R(a,e)=R(d,e)=\{a,d,e\}$,
  $R(c,d)=R(c,f)=R(d,f)=\{c,d,f\}$ and all other sets are singletons or
  $V$. Here $R$ is monotone and satisfies \AX{(i)} but the sets
  $\{a,b,c\}$, $\{a,d,e\}$, $ \{c,d,f\}$ violate \AX{(wp)}.
\end{example}
Moreover, axioms \AX{(i)} and \AX{(w)} are independent. The canonical
transit function in Fig.~\ref{fig:cex1}B satisfies \AX{(i)} but violates
\AX{(w)} and in Fig.~\ref{fig:cex1}C, it satisfies \AX{(w)} but violates
\AX{(i)}.

\begin{lemma}
  Suppose $\mathscr{C}$ is a weak hierarchy. Then \AX{(I)} implies
  \AX{(WP)}.
\end{lemma}
\begin{proof}
  Suppose $A,B,C\in \mathscr{C}$ intersect pairwisely. Then $A\cap B\cap
  C\ne\emptyset$ and we may assume, w.l.o.g., $A\cap B\cap C = A\cap B$,
  and thus $A\cap B\subseteq C$. Then either $C\setminus(A\cup
  B)=\emptyset$, i.e., $C\subseteq A\cup B$, or \AX{(I)} implies
  $A\subseteq C$ and thus also $A\subseteq B\cup C$ or $B\subseteq C$ and
  thus also $B\subseteq A\cup C$. In either case, therefore, one of the
  three sets is contained in the union of the other two, and thus
  $\mathscr{C}$ satisfies \AX{(WP)}.
\end{proof}
\begin{corollary}
  If $\mathscr{C}$ is a weak hierarchy, then \AX{(I)} and \AX{(WP)} are
  equivalent.
\end{corollary}
A transit function property that is weaker than the axiom \AX{(u)} and 
proved to be weaker than the axiom \AX{(wp)} in \cite{Changat:21w} is:
\begin{description}
\item[\AX{(o')}] For all $u,v\in V$ and $z\in R(u,v)$ there exist $p,q\in
  R(u,v)$ such that $R(p,z)\cup R(z,q)=R(u,v)$.
\end{description}
We next show that \AX{(i)} is in general weaker than \AX{(wp)} and stronger
than \AX{(o')}:
\begin{lemma}\label{i*->o'}
  If a monotone transit function $R$ satisfies \AX{(i)}, then it also
  satisfies \AX{(o')}.
\end{lemma}
\begin{proof}
  Let $R$ be a monotone transit function satisfying \AX{(i)} and suppose,
  for contradiction, that $R$ violates \AX{(o')}. Then there exist
  $u,v,z\in V$ such that $z\in R(u,v)$ and $R(u,v)\nsubseteq R(u_i,z)\cup
  R(u_j,z)$ for all $u_i,u_j\in R(u,v)$.  Consider $u_i\neq u_j\neq
  z$. Then there exists $u_{k_1}\in R(u,v)$ such that $u_{k_1}\notin
  R(u_i,z)\cup R(u_j,z)$. We distinguish two cases:
  \par\noindent\textit{Case 1:} One of the three sets $R(u_i,z)$,
  $R(u_j,z)$, $R(u_{k_1},z)$ is contained in another one, i.e.,
  $R(u_i,z)\subset R(u_{k_1},z)$ or $R(u_j,z)\subset R(u_{k_1},z)$ or
  $R(u_i,z) \subseteq R(u_j,z)$ or $R(u_j,z)\subseteq R(u_i,z)$. In this
  case, we consider a new point $u_{k_2}\in R(u,v)$ that is not in any of
  the three sets. Such a point exists since $R$ violates \AX{(o')}.  If
  $R(u_i,z)\subset R(u_{k_1},z)$ holds, then consider the sets $R(u_j,z)$,
  $R(u_{k_1},z)$, and $R(u_{k_2},z)$. If at least one of the three sets is
  contained in another, we consider a new point $u_{k_3}\in R(u,v)$ which
  is not in any of the three sets. Again, such a point exists since $R$
  violates \AX{(o')}. Continuing in this manner, we obtain an infinite
  number of points $u_{k_1},u_{k_2},\dots$ because in each step, a new point
  from $R(u,v)$ is added. However, this contradicts the fact that $R(u,v)$
  is finite. After a finite number of steps, we, therefore, encounter
  \par\noindent\textit{Case 2:} $R(u_i,z)$, $R(u_j,z)$, $R(u_k,z)$ overlap
  pairwisely. Suppose the intersection of two sets is contained in the
  third, say, $R(u_i,z)\cap R(u_j,z)\subset R(u_k,z)$; then \AX{(i)}
  implies $R(u_j,z)\subset R(u_k,z)$ or $R(u_i,z)\subset R(u_k,z)$
  contradicting the assumption that the three sets overlap
  pairwisely. Therefore we have $R(u_i,z)\cap R(u_j,z)\nsubseteq R(u_k,z)$,
  $R(u_i,z)\cap R(u_k,z)\nsubseteq R(u_j,z)$, and $R(u_k,z)\cap
  R(u_j,z)\nsubseteq R(u_i,z)$. Hence there exist points $v_1,v_2,v_3\in
  R(u,v)$ such that $v_1 \in R(u_i,z)\cap R(u_j,z)$, $v_1\notin R(u_k,z)$,
  $v_2\in R(u_i,z)\cap R(u_k,z)$, $v_2\notin R(u_j,z)$, and $v_3\in
  R(u_k,z)\cap R(u_j,z)$, $v_3\notin R(u_i,z)$. Now consider the sets
  $R(v_1,z)$, $R(v_2,z)$, and $R(v_3,z)$. Since the sets $R(u_i,z)$,
  $R(u_j,z)$, and $R(u_k,z)$ overlap pairwisely, by \AX{(m)}, the sets
  $R(v_1,z)$, $R(v_2,z)$, and $R(v_3,z)$ also overlap pairwisely.  If the
  intersection of two sets is contained in the third one, then \AX{(i)}
  again implies that one of the sets is contained in another,
  contradicting the assumption that the three sets overlap pairwisely.
  Then there exist points $w_1,w_2,w_3\in R(u,v)$ such that $w_1 \in
  R(v_1,z)\cap R(v_2,z)$, $w_1\notin R(v_3,z)$, $w_2\in R(v_1,z)\cap
  R(v_3,z)$, $w_2\notin R(v_2,z)$, and $w_3\in R(v_2,z)\cap R(v_3,z)$,
  $w_3\notin R(v_1,z)$. Repeating these arguments, we eventually obtain
  points $x_1,x_2,x_3\in R(u,v)$ such that $R(x_1,z)\cap R(x_2,z)=\{z\}$,
  $R(x_2,z)\cap R(x_3,z)=\{z\}$, and $R(x_1,z)\cap R(x_3,z)=\{z\}$.  The
  three sets $R(x_1,z), R(x_2,z), R(x_3,z)$ violate \AX{(i)},
  contradicting the assumption that \AX{(i)} is satisfied.  Thus $R$ must
  satisfy \AX{(o')}.
\end{proof}

Example~\ref{o-i*} below shows that the converse need not be true:
\begin{example}\label{o-i*}
  Let $V=\{a,b,c,d\}$ and $R$ on $V$ be defined by $R(a,b)=\{a,b\}$,
  $R(b,c)=\{b,c\}$, $R(b,d)=\{b,d\}$ and all other sets are singletons or
  $V$.  $R$ is monotone, satisfies \AX{(o')} but not \AX{(i)}.
\end{example}

\section{Between paired hierarchies and pyramids}
\label{sect:l1l2}

Consider a set system $\mathscr{C}$ satisfying the following properties:
\begin{description}\setlength{\itemsep}{0pt}\setlength{\parskip}{0pt}%
\item[\AX{(L1)}] Let $A,B,C\in\mathscr{C}$. If $A\overlaps B$ and
  $B\overlaps C$, then (i) $A\subseteq C$, or (ii) $C\subseteq A$, or (iii)
  $A\cap C\subseteq B\subseteq A\cup C$.
\item[\AX{(N3O)}] Let $A,B,C\in\mathscr{C}$. If $A\overlaps B$ and
  $B\overlaps C$ then $A$ does not overlap $C$.  
\end{description}
Axiom \AX{(N3O)} appeared in recent work on the clustering systems of
so-called galled trees, a special class of level-1 phylogenetic networks
\cite{Hellmuth:22q}.  By definition, a paired hierarchy trivially satisfies
\AX{(N3O)} and \AX{(L1)} since, in this case, we must have $A=C$. Moreover,
these axioms are related as follows:
\begin{lemma}
  \label{lem:L1->WP}
  If $\mathscr{C}$ satisfies \AX{(L1)}, then $\mathscr{C}$ satisfies
  \AX{(WP)}.
\end{lemma}
\begin{proof}
  Suppose $A,B,C\in\mathscr{C}$ pairwisely intersect. Then either
  $A\overlaps B$ and $B\overlaps C$, or $B$ and at least one of $A$ and $C$
  are nested. In the latter case, we may assume w.l.o.g.\ $B\subseteq C$ or
  $C\subseteq B$, and thus $B\subseteq A\cup C$ or $C\subseteq A\cup B$. If
  $A\overlaps B$ and $B\overlaps C$ then \AX{(L1)} implies (i) $A\subseteq
  C\subseteq B\cup C$, or (ii) $C\subseteq A\subseteq A\cup C$, or (iii)
  $B\subseteq A\cup C$. In either case, one of $A$, $B$, $C$ is contained
  in the union of the other two sets.
\end{proof}

\begin{lemma}
  \label{lem:L1->W}
  Let $\mathscr{C}$ be a set system satisfying \AX{(L1)}. Then
  $\mathscr{C}$ is a weak hierarchy.
\end{lemma}
\begin{proof}
  Let $A,B,C\in\mathscr{C}$. First, suppose $A$, $B$, and $C$ do not
  overlap. Then either $A, B, C$ are pairwise disjoint, in which case $A\cap
  B\cap C=\emptyset =A\cap B$ or the three sets are nested. Assume,
  w.l.o.g., that $A\subseteq B\subseteq C$. Then $A\cap B\cap C= A\cap B=
  A$. Second, assume $A\overlaps B$ and $B$ does not overlap $C$. The
  following four situations may occur: (a) $B\cap C=\emptyset$ implies
  $A\cap B\cap C=\emptyset = B\cap C$. (b) $B\subseteq C$, implies $A\cap
  B\cap C= A\cap B$, and (c) $C\subseteq B$ implies $A\cap B\cap C=A\cap
  C$. In either sub-case, $A\cap B\cap C\in\{A\cap B, A\cap C, B\cap
  C\}$. By symmetry, the same is true is $A$ does not overlap $B$ and
  $B\overlaps C$.  Finally, suppose $A\overlaps B$ and $B\overlaps
  C$. According to \AX{(L1)}, three cases may occur. In case (i),
  $A\subseteq C$ implies $A\cap B\cap C= A\cap B$. In case (ii), we have
  $A\cap B\cap C= C\cap B$. In case (iii), $A\cap C\subseteq B$ implies
  $A\cap B\cap C=A\cap C$. Thus, in either case, $A\cap B\cap C\in\{A\cap B,
  A\cap C, B\cap C\}$.
\end{proof}

\begin{corollary}
  If $\mathscr{C}$ satisfies \AX{(L1)}, then it is weakly pyramidal.
\end{corollary}

\begin{lemma}
  Let $\mathscr{C}$ be a set system. Then \AX{(L1)} implies \AX{(N3O)}.
  \end{lemma}
\begin{proof}
  Suppose axiom \AX{(L1)} holds, $A\overlaps B$ and $B\overlaps C$, and we
  have neither $A\subseteq C$ nor $C\subseteq A$. In case (iii), one of the
  following situations are possible: (a) $A \cap C=\emptyset$, (b) $B=A\cap
  C$, (c) $B= A\cup C$, (d) $\emptyset\ne A\cap C \subsetneq B \subsetneq
  A\cup C$. In case (b), we have $B\subseteq A$ and $B\subseteq C$, and
  thus $B$ does not overlap $A$ and $C$, a contradiction. In case (c),
  $A\subseteq B$ and $C\subseteq B$ imply that $B$ does not overlap $A$ and
  $C$, again a contradiction. In case (d), we have $A\overlaps C$.  By
  \AX{(L1)}, $A\overlaps B$ and $A\overlaps C$ imply that $B\cap C\subseteq
  A \subseteq B\cup C$, and thus $A\cap B\cap C=B\cap C$, since the other
  two options of \AX{(L1)} contradict the assumption that $B$ and $C$
  overlap.  By the same argument, $A\overlaps C$ and $B\overlaps C$ and
  \AX{(L1)} implies $A\cap B\cap C= A\cap B$. As in the proof of
  Lemma~\ref{lem:L1->W}, $A\overlaps B$ and $B\overlaps C$ implies $A\cap
  B\cap C=A\cap C$. Thus, in case (iii), we have either $A\cap C=\emptyset$
  or $A\overlaps C$, in which case $A\cap B=A\cap C=B\cap C$ must
  hold. Therefore, $B\subseteq A\cup C$ implies $B=B\cap(A\cup C) = (A\cap
  B)\cup(C\cap B) = (B\cap C)$, contradicting $B\overlaps C$. Therefore,
  $A\cap C=\emptyset$.
\end{proof}
The following Example~\ref{ex:N3O-L1} shows that the converse is not true.
\begin{example}
  \label{ex:N3O-L1}
  Let $A=\{a,b\}, B=\{b,c,d\}, C=\{d,e\}, V$ and the singletons be the sets
  in a set system $\mathscr{C}$ on $V=\{a,b,c,d,e\}$. Then, $\mathscr{C}$
  satisfies \AX{(N3O)} but violates \AX{(L1)} as $B\nsubseteqq A\cup C$.
\end{example}
We can translate \AX{(L1)} and \AX{(N3O)} in terms of transit functions as
follows:
\begin{description}\setlength{\itemsep}{0pt}\setlength{\parskip}{0pt}%
\item[\AX{(l1)}] $R(x,y)\overlaps R(p,q)\overlaps R(u,v)$ implies either
  (i) $R(x,y)\subseteq R(u,v)$ or (ii) $R(u,v)\subseteq R(x,y)$ or (iii)
  $R(x,y)\cap R(u,v) \subseteq R(p,q) \subseteq R(x,y)\cup R(u,v)$.
\item[\AX{(n3o)}] $R(x,y)\overlaps R(p,q)$ and $R(p,q)\overlaps R(u,v)$
  implies either (i) $R(x,y)\subseteq R(u,v)$ or (ii) $R(u,v)\subseteq
  R(x,y)$, or (iii) $R(x,y)\cap R(u,v)=\emptyset$.
\end{description}
\begin{corollary}
  \label{cor:l1->n3o}
  If a monotone transit function $R$ satisfies \AX{(l1)}, then $R$
  satisfies \AX{(n3o)}.
\end{corollary}
\begin{lemma}
  \label{lem:n3o->wpy}
  If a monotone transit function $R$ satisfies \AX{(n3o)}, then it also
  satisfies \AX{(w)} and \AX{(wp)}, i.e., $R$ is weakly pyramidal.
\end{lemma}
\begin{proof}
  Suppose $R$ does not satisfy \AX{(w)}. Then there exist points
  $x,y,z\in V$ such that $x\notin R(y,z)$, $y\notin R(x,z)$, and $z\notin
  R(x,y)$, and thus the three transit sets overlap pairwisely, i.e., we
  have $R(x,y)\overlaps R(y,z)$, $ R(y,z)\overlaps R(z,x)$, and
  $R(z,x)\overlaps R(x,y)$, violating \AX{(n3o)}.
  
  Now suppose $R$ satisfies \AX{(n3o)} and consider three transit sets
  $A,B,C\in \mathscr{C}_R$ with pairwise nonempty intersections. If at
  least one set is contained in another set, then $R$ satisfies
  \AX{(wp)}. Otherwise, $A\overlaps B$, $B\overlaps C$, and $A\overlaps C$ 
  contradict \AX{(n3o)}.
\end{proof}
Nevertheless, \AX{(n3o)} is not sufficient to guarantee that $R$ is
pyramidal, which is clear from Fig.~\ref{fig:cex1}A. Moreover, Example
\ref{py-n3o} shows that the converse of Lemma~\ref{lem:n3o->wpy} is also
not true.
\begin{example}\label{py-n3o}
  Let $R$ be defined on $V=\{a,b,c,d,e\}$ by $R(a,b)=\{a,b,c\}=R(a,c)$,
  $R(a,d)=R(a,e)=R(b,e)=V$, $R(b,c)=\{b,c\}$, $R(b,d)=\{b,c,d\}$,
  $R(c,d)=\{c,d\}$, and $R(c,e)=R(d,e)=\{c,d,e\}$. The transit function $R$
  is pyramidal but does not satisfy \AX{(n3o)}.
\end{example}
A monotone transit function $R$ satisfying \AX{(l1)} satisfies \AX{(w)}
and, therefore, also \AX{(k2)}. Consequently, we have
\begin{fact}
  If the monotone transit function $R$ satisfies \AX{(l1)}, then
  $\mathscr{C}_R$ is closed.
\end{fact}

Let us now consider the following axiom, which is related to, but weaker
than, the union-closure condition \AX{(uc)}:
\begin{description}\setlength{\itemsep}{0pt}\setlength{\parskip}{0pt}%
\item[{\AX{(l2)}}] If $R(x,y)\overlaps R(p,q)\overlaps R(u,v)$ with
  $R(x,y)\cap R(u,v)= \emptyset$, then $R(x,y)\cup R(p,q)\cup R(u,v)=
  R(s,t)$ for some $s\in R(x,y)\setminus R(p,q)$ and $t\in R(u,v)\setminus
  R(p,q)$.
\end{description}

If $R$ is a monotone transit function such that $\mathscr{C}_R$ is a paired
hierarchy, then $R$ satisfies \AX{(l2)}. In the language of set
systems, \AX{(l2)} clearly implies the following property:
\begin{description}\setlength{\itemsep}{0pt}\setlength{\parskip}{0pt}%
\item[{\AX{(L2')}}] Let $A,B,C\in\mathscr{C}$, $A\overlaps B$, $B\overlaps
  C$, and $A\cap C=\emptyset$. Then $A\cup B\cup C\in\mathscr{C}$.
\end{description}
However, the following Example~\ref{ex:w+wp+l1-l2} shows that \AX{(L2')} is not the
correct ``translation'' of \AX{(l2)}:
\begin{example}
  \label{ex:w+wp+l1-l2}
  $R$ on $V=\{a,b,c,d,e,f\}$ defined by $R(a,b)=\{a,b\}$, $R(a,c)=\{a,c\}$,
  $R(a,d)=R(a,e)=R(b,e)=R(c,e)=\{a,b,c,d,e\}$, $R(a,f)=R(c,f)=\{a,c,f\}$,
  $R(b,c)=R(b,d)=R(c,d)=\{b,c,d\}$, $R(b,f)=V$, $R(d,f)=R(e,f)=\{d,e,f\}$,
  $R(d,e)=\{d,e\}$ is monotone. $\mathscr{C}_R$ satisfies \AX{(L2')}. But
  $R$ violates \AX{(l2)} as $R(a,b)\overlaps R(b,d)\overlaps R(d,f)$ but
  $R(a,e)\subsetneq R(a,b)\cup R(b,d)\cup R(d,f)$ and $R(a,f)\subsetneq
  R(a,b)\cup R(b,d)\cup R(d,f)$.
\end{example}
Condition \AX{(l2)} on transit functions is more restrictive than
\AX{(L2')} since, in addition to the behavior of the transit sets, it also
requires the existence of two ``reference points'' $s$ and $t$ that satisfy
an additional condition. An axiom for $R$ corresponding to \AX{(L2')}
would require only the existence of $s$ and $t$ such that $R(x,y)\cup
R(p,q)\cup R(u,v)= R(s,t)$.  A related property for general set systems is :
\begin{description}\setlength{\itemsep}{0pt}\setlength{\parskip}{0pt}%
\item[\AX{(L2'')}] Let $A,B,C,D\in\mathscr{C}$, $A\overlaps B\overlaps C$,
  $A\cap C=\emptyset$ and $(A\cup C)\setminus B\subseteq D$.  Then $A\cup
  B\cup C\subseteq D$.
\end{description}

\begin{lemma}
  If $R$ is a monotone transit function satisfying \AX{(w)} and
  \AX{(l2)}, then $\mathscr{C}_R$ is a weak hierarchy satisfying
  \AX{(L2')} and \AX{(L2'')}.
\end{lemma}
\begin{proof}
  Suppose $R$ satisfies \AX{(w)} and \AX{(l2)}. Then $\mathscr{C}_R$ is
  obviously a weak hierarchy satisfying \AX{(L2')}. Set $A=R(x,y)$,
  $B=R(p,q)$, and $C=R(u,v)$, i.e., $A,B,C\in\mathscr{C}_R$ and assume
  $A\overlaps B\overlaps C$ and $A\cap C=\emptyset$. Then by \AX{(l2)}, we
  have $A\cup B\cup C\in \mathscr{C}_R$. Let $D=R(s',t')\in\mathscr{C}_R$
  such that $(A\cup C)\setminus B=(R(x,y)\setminus R(p,q))\cup
  (R(u,v)\setminus R(p,q))\subseteq D$. Then by \AX{(l2)}, there exist $s\in
  R(x,y)\setminus R(p,q)$ and $t\in R(x,y)\setminus R(p,q)$ such that
  $R(s,t)=A\cup B\cup C$, and by monotonicity, $R(s,t)\subseteq D$,
  i.e., $A\cup B\cup C\subseteq D$, and thus \AX{(L2'')} is satisfied.
\end{proof}

\begin{figure}[t]
  \begin{center}
    \includegraphics[width=0.9\textwidth]{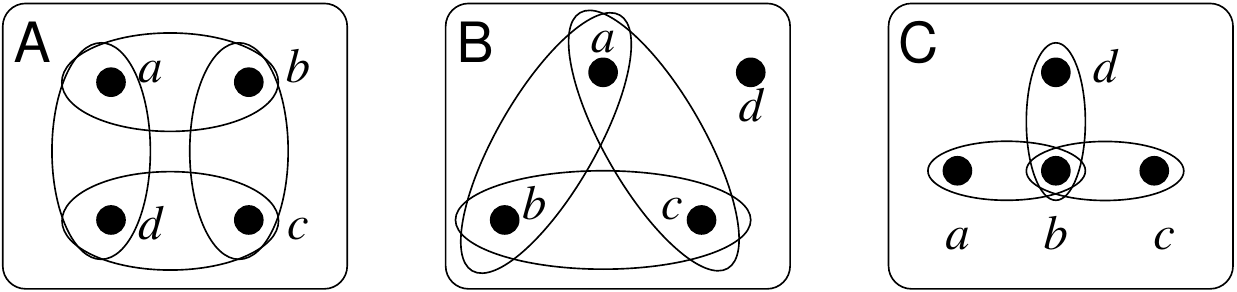}
  \end{center}
  \caption{Three set systems with corresponding canonical transit functions
    that serve as counter-examples, ruling out potential implications
    between axioms in this contribution: (\textsf{A}) The set system
    $\mathscr{C}$ on $V$ comprising the singletons, $V$, and the four edges
    $C_1=\{a,b\}$, $C_2=\{b,c\}$, $C_3=\{c,d\}$ and $C_4=\{d,a\}$ is a weak
    hierarchy (since any triple of sets that intersect pairwise contains
    either $V$ or a singleton), and satisfies axiom \AX{(WP)}. Since the
    four edges form a 4-cycle $(C_1,C_2,C_3,C_4)$ in $(V,\mathscr{C})$,
    there is no linear ordering on $V$ compatible with $\mathscr{C}$. Thus
    $(V,\mathscr{C})$ is weakly (pre-)pyramidal but not (pre-)pyramidal.
    (\textsf{B}) The set system $\mathscr{C}$ comprises the singletons,
    $V$, and the three pairs $\{p,q\}$ with $p,q\in\{a,b,c\}$. These three
    pairs intersect pairwisely, but $\{a,b\}\cap\{a,c\}\cap
    \{b,c\}=\emptyset$; thus, $\mathscr{C}$ is not a weak hierarchy.  The
    canonical transit function satisfies $R(p,q)=\{p,q\}$ for
    $p,q\in\{a,b,c\}$ and $R(p,d)=V$ for $p\in\{a,b,c\}$. (\textsf{C}) The
    set system $\mathscr{C}$ consists of the singletons, the three edges
    $\{a,b\}$, $\{b,c\}$, $\{b,d\}$, and $V$. It is a weak hierarchy. Its
    canonical transit function $R$ satisfies $R(p,q)=V$ if and only $p\ne
    q$ and $p,q\in\{a,c,d\}$.}
  \label{fig:cex1}
\end{figure}

In the following Example~\ref{uc+l2-l1}, we see the independence of axioms
\AX{(l1)} and \AX{(l2)} and their connection with the axioms \AX{(w)},
\AX{(wp)}, and \AX{(uc)}:
\begin{example}\label{uc+l2-l1}
  Let $R$ be a monotone transit function on $V=\{a,b,c,d,e,f\}$ given by
  $R(a,b)=\{a,b,c\}$, $R(a,c)=\{a,b,c\}$, $R(a,d)=\{a,b,c,d,e\} = R(a,e)$,
  $R(b,c)=\{a,b,c\}$, $R(b,d)= R(b,e) = \{b,c,d,e\}$,
  $R(b,f)=\{b,c,d,e,f\}$, $R(c,d)= R(c,e) =\{c,e,d\}$, $R(c,f) =
  \{c,f,d,e\}$, $R(d,e)=\{c,d,e\}$, $R(d,f)=\{c,d,e,f\}$, $R(e,f)=\{e,f\}$,
  and $R(a,f) = V$.  Here $R$ satisfies \AX{(l2)}, but violates \AX{(l1)}
  since $R(a,b)=\{a,b,c\}$, $R(c,d)=\{c,e,d\}$, and $R(e,f)=\{e,f\}$.  The
  corresponding transit set $\mathscr{C}_R$ satisfies \AX{(UC)}.
\end{example}
\begin{example}\label{l1+l2-uc}
  Let $R$ be a monotone transit function on $V=\{a,b,c,d,e\}$ given by
  $R(a,b)=\{a,b,c,d\}$, $R(a,c)=\{a,c\}$, $R(a,d)=\{a,b,c,d\}$,
  $R(b,c)=\{b,c\}$, $R(b,d)=\{b,d\}$, $R(b,e)=\{b,e,c,d\}$,
  $R(c,d)=\{c,b,d\}$, $R(c,e)= R(d,e)= \{c,b,d,e\}$ and all other
  non-singleton sets are $V$. Then $R$ satisfies \AX{(l1)} and \AX{(l2)}
  and $\mathscr{C}_R$ is pyramidal. However, $R$ does not satisfy \AX{(uc)}
  because $R(a,c)\cup R(b,c)\notin \mathscr{C}_R$.
\end{example}    
\begin{example}
  \label{ex:l2-w-wp}
  Let $R$ on $V=\{a,b,c,d,e\}$ be defined by $R(a,b)=\{a,b\}$,
  $R(a,c)=\{a,c\}$, $R(a,d)=V$, $R(b,c)=\{a,b,c\}$,
  $R(b,d)=R(c,d)=R(d,e)=V$, $R(a,e)=\{a,e\}$, $R(b,e)=\{a,b,e\}$, and
  $R(c,e)=\{a,c,e\}$. Here $R$ satisfies \AX{(l2)} but violates both
  \AX{(w)} and \AX{(wp)}.
\end{example}
\begin{example}
  \label{ex:py+l1+u3-l2-ph}
  Let $R$ on $V=\{a,b,c,d,e\}$ be defined by $R(a,b)=\{a,b\}$,
  $R(b,c)=\{b,c\}$, $R(c,d)=\{c,d\}$, and other non-singleton transit sets
  equal $V$. Once checks that $R$ is pyramidal but violates \AX{(l2)}.
  Also, $R$ satisfies \AX{(l1)}, and $\mathscr{C}_R$ is not a paired
  hierarchy.
\end{example}
\begin{example}
  \label{ex:py+l2-l1-ph}
  Let $R$ on $V=\{a,b,c,d,e\}$ be defined by $R(a,b)=\{a,b\}$,
  $R(b,c)=\{b,c,d\}=R(c,d)=R(b,d)$, $R(d,e)=\{d,e\}$, and all other
  non-singleton transit sets equal $V$.  Here $R$ is pyramidal and
  satisfies \AX{(l2)} but violates \AX{(l1)}.  Moreover, $\mathscr{C}_R$ is
  not a paired hierarchy.
\end{example}
\begin{example}
  \label{Ex:py-l1-l2}
  Let $R$ on $V=\{a,b,c,d,e,f\}$ be defined by $R(a,b)=\{a,b\}$,
  $R(b,c)=\{b,c,d\}= R(c,d)= R(b,d)$, $R(d,e)=\{d,e\}$ and all other
  non-singleton transit sets equal $V$. Here $R$ is pyramidal but violates
  both \AX{(l1)} and \AX{(l2)}.
\end{example}
\begin{example}
  \label{Ex:L2+wp-l1-w}
  Let $R$ on $X=\{a,b,c,d\}$ be defined by $R(a,b)=\{a,b\}$,
  $R(b,c)=\{b,c\}$, $R(a,c)=\{a,c\}$, and other sets are singletons and
  $X$. Here $R$ satisfies \AX{(l2)} and \AX{(wp)}, violates \AX{(l1)} and
  \AX{(w)}.
\end{example}
From Examples~\ref{uc+l2-l1} and \ref{ex:py+l1+u3-l2-ph}, we see that
\AX{(l1)} and \AX{(l2)} are independent.  Examples~\ref{uc+l2-l1} and
\ref{l1+l2-uc} show that \AX{(uc)} and \AX{(l1)} are independent.
Furthermore, Examples~\ref{ex:l2-w-wp} and \ref{ex:py+l1+u3-l2-ph} show
that \AX{(l2)} is independent of both \AX{(w)} and \AX{(wp)}.  The transit
function $R$ in Example~\ref{ex:w+wp+l1-l2}, furthermore, satisfies
\AX{(w)}, \AX{(wp)}, and \AX{(l1)} but violates \AX{(l2)}, and its transit set
$\mathscr{C}_R$ is not pyramidal.

Axiom \AX{(l2)} can be seen as a relaxation of the union-closure property.
Indeed, we have
\begin{lemma}
  Let $R$ be a monotone transit function satisfying \AX{(uc)}, then $R$
  satisfies \AX{(l2)}.
\end{lemma}
\begin{proof}
  Let $R(x,y)\overlaps R(p,q)\overlaps R(u,v)$ with $R(x,y)\cap R(u,v)=
  \emptyset$. Then by \AX{(uc)}, $R(x,y)\cup R(p,q)= R(s_1,t_1)$ for some
  $s_1\in R(x,y)\setminus R(p,q)$ and $t_1\in R(p,q)\setminus R(x,y)$ and
  $R(p,q)\cup R(u,v)= R(s_2,t_2)$ for some $s_2\in R(p,q)\setminus R(u,v)$
  and $t_2\in R(u,v)\setminus R(p,q)$. Since $R(s_1,t_1)\overlaps
  R(s_2,t_2)$ we obtain $R(s_1,t_1)\cup R(s_2,t_2)= R(s,t)$ for some $s\in
  R(s_1,t_1)\setminus R(s_2,t_2)$ and $t\in R(s_2,t_2)\setminus
  R(s_1,t_1)$.  That is, $R(x,y)\cup R(p,q)\cup R(u,v)= R(s,t)$ for some
  $s\in R(x,y)\setminus R(p,q)$ and $t\in R(u,v)\setminus R(p,q)$. Hence
  $R$ satisfies \AX{(l2)}.
\end{proof}

\begin{lemma}
  \label{lem:l2-no-hypercycle}
  If $R$ is a monotone transit function satisfying \AX{(l1)} and \AX{(l2)},
  then $\mathscr{C}_R$ is totally balanced ($\beta$-acyclic).
\end{lemma}
\begin{proof}
  Since \AX{(l1)} implies \AX{(w)}, $\mathscr{C}_R$ cannot contain a
  3-cycle. Also, it follows from \AX{(l1)} that $\mathscr{C}_R$ can not
  contain a weak $\beta$-cycle $(C_1,\dots,C_n)$ with $(C_i\setminus
  C_{i+1}) \cup (C_i\setminus C_{i-1}) \neq\emptyset$ for some $i$. First,
  suppose that $\mathscr{C}_R$ contains a weak $\beta$-cycle of length 4
  say, $(C_1,C_2,C_3,C_4)$ with $(C_i\setminus C_{i+1}) \cup (C_i\setminus
  C_{i-1}) =\emptyset$ for all $i$. We have  $C_1\overlaps C_2\overlaps
  C_3$. Since it is a weak $\beta$-cycle, $C_1\nsubseteqq C_3$ and
  $C_3\nsubseteqq C_1$. Therefore, using that \AX{(L1)} implies \AX{(N3O)}
  we have, $C_1\cap C_3=\emptyset$. Then \AX{(l2)} implies that
  $R(p,q)=C_1\cup C_2\cup C_3$ for $p\in C_1\setminus C_2$ and $q\in
  C_3\setminus C_2$. Moreover, $p,q\in C_4$ and therefore, $R(p,q)\subseteq
  C_4$, a contradiction.  Hence $\mathscr{C}_R$ can not contain such a
  4-cycle. Now assume that $\mathscr{C}_R$ contains a weak $\beta$-cycle
  $(C_1,\dots ,C_n)$ of length $n>4$. As argued above, it must satisfy
  $(C_i\setminus C_{i+1}) \cup (C_i\setminus C_{i-1}) =\emptyset$ for all
  $i$. Since $C_{i-1}\overlaps C_i \overlaps C_{i+1}$ and
  $C_{i-1}\nsubseteqq C_{i+1}$ and $C_{i+1}\nsubseteqq C_{i-1}$, we have
  $C_{i-1}\cap C_{i+1}=\emptyset$ by \AX{(N3O)}.  In particular, we have
  $C_1\overlaps C_2\overlaps C_3$ and $C_1\cap C_3=\emptyset$. Together
  with \AX{(l2)}, this implies that $C_1\cup C_2\cup C_3\eqqcolon C \in
  \mathscr{C}_R$. Then the weak $\beta$-cycle, $(C,C_4,\dots ,C_n)$ has the
  property that $(C\setminus C_{4})\cup (C\setminus C_{n}) \neq \emptyset$,
  which is the case we have already ruled out. Hence $\mathscr{C}_R$ is
  totally balanced ($\beta$-acyclic).
\end{proof}

\begin{theorem}
  \label{thm:L1L2->PY}
  If $R$ is a monotone transit function satisfying \AX{(l1)} and \AX{(l2)},
  then $\mathscr{C}_R$ is pyramidal.
\end{theorem}
\begin{proof}
  Since \AX{(l1)} implies that $\mathscr{C}_R$ satisfies \AX{(L1)}, we
  conclude from Obs.~\ref{lem:L1->WP} and Lemma~\ref{lem:L1->W} that
  $\mathscr{C}_R$ is closed and satisfies \AX{(W)} and
  \AX{(WP)}. Furthermore, \AX{(L1)} implies that $\mathscr{C}_R$ cannot
  contain the second, third, and fourth forbidden configurations of an
  interval hypergraph in Eq.(\ref{eq:inthypg}).  Since $\mathscr{C}$ is
  closed, the fifth configuration is also ruled out because the
  intersection of the large sets in the fifth configuration is again a
  cluster. Thus, the fifth configuration contains the fourth one as a
  subhypergraph. A closed clustering system satisfying \AX{(L1)} is,
  therefore, either pyramidal or contains a hypercycle
  $\mathscr{H}_n\coloneqq \{C_i|1\le i\le n\}$ such that
  $\mathcal{C}_i\setminus (\mathcal{C}_{i+1} \cup \mathcal{C}_{i-1})
  =\emptyset$ for all $1\le i\le n$ (indices taken $\mod n$), for some $n>3$. By Lemma~\ref{lem:l2-no-hypercycle}, \AX{(l1)} and \AX{(l2)}
    rule out the existence of such a hypercycle.
\end{proof}
Axiom \AX{(L1)} thus enforces the existence of a linear order locally. It
is insufficient to ensure global consistency with a linear order, however.  

As a consequence of Thm.~\ref{thm:L1L2->PY} and Example~\ref{Ex:py-l1-l2},
the transit functions satisfying \AX{(l1)} and \AX{(l2)} are a proper
subset of the pyramidal transit functions, which are a proper
subset of the weakly pyramidal transit functions. On the other hand,
monotone transit functions of paired hierarchies satisfy \AX{(l1)} and
\AX{(l2)}. Example~\ref{l1+l2-uc} again shows that the converse is not true.

\section{Totally Balanced Transit Functions}
\label{sect:tb}

In \cite{Changat:18a} we considered \AX{(u)} as a key property of
hierarchies. Fig.~4 in \cite{Changat:21w} shows, however, that pyramidal
transit functions do not necessarily satisfy \AX{(u)}. Here we consider a
natural generalization:
\begin{description}
\item[\AX{(u3)}] If $R(x,y)\not\subseteq\{x,y\}$, then there exists  $z\in
  R(x,y)\setminus\{x,y\}$ such that $R(x,z)\cup R(z,y)=R(x,y)$.
\end{description}
Clearly, \AX{(u)} implies \AX{(u3)}. Example~\ref{ex:ph-not-uc} shows that
the converse is not true.

The example of Fig.~\ref{fig:cex1}A indicates that there are monotone
transit functions that satisfy \AX{(w)} and \AX{(wp)} but violate
\AX{(u3)}. Example~\ref{ex:l2-w-wp}, furthermore, shows that \AX{(u3)} does
not imply \AX{(wp)}. Fig.~\ref{fig:cex1}C shows that axioms \AX{(w)} and
\AX{(u3)} do not imply \AX{(wp)}, and Fig.~\ref{fig:cex1}B shows that
satisfying \AX{(wp)} and \AX{(u3)} does imply \AX{(w)}.  Hence, \AX{(w)},
\AX{(wp)}, and \AX{(u3)} are independent.

\begin{lemma}\label{py->u3}
  If $R$ is pyramidal, then $R$ satisfies \AX{(u3)}.
\end{lemma}
\begin{proof}
  Assume that $R$ is pyramidal and assume for contradiction that $R(x,y)$
  violates \AX{(u3)}. Then, for every $z\in R(x,y)\setminus\{x,y\}$, we
  have $R(x,z)\cup R(z,y)\subsetneq R(x,y)$, and thus there exists $z'\in
  R(x,y)\setminus (R(x,z)\cup R(z,y))$. Consider two such elements $z$ and
  $z'\in R(x,y)$. Using again that $R(x,y)$ violates \AX{(u3)}, we have
  $R(x,z')\cup R(z',y)\subsetneq R(x,y)$ and thus $R(x,z)\overlaps R(z,y)$
  and $R(x,z')\overlaps R(z',y)$. If $z\notin R(x,z')\cup R(z',y)$, then
  $R(x,z)$, $R(z,y)$, $R(z',y)$, and $R(x,z')$ form a hyper-cycle, i.e.,
  the first forbidden configuration in Eq.(\ref{eq:inthypg}), thus
  contradicting that $R$ is pyramidal. Therefore, $z\in R(x,z')\cup
  R(z',y)$.  Now suppose $z\in R(x,z')\setminus R(z',y)$ and note that
  $y\in R(z',y)\setminus R(x,z')$. Since pyramidal set systems are in
  particular weak hierarchies, we can use \AX{(w3)} to infer $z'\in
  R(z,y)$, a contradiction. If $z\in R(z',y)\setminus R(x,z')$, again by
  using \AX{(w3)}, we obtain $z'\in R(x,z)$, a contradiction. Therefore, we
  conclude that $z\in R(x,z')\cap R(z',y)$, which implies $R(x,z)\cup
  R(z,y)\subset R(x,z')\cup R(z',y)$.
  
  Since $R(x,z')\cup R(z',y)\subsetneq R(x,y)$ there exists $z''\in
  R(x,y)\setminus R(x,z')\cup R(z',y)$.  Moreover, $R(x,z')\overlaps
  R(z',y)$ and $R(x,z'')\overlaps R(z'',y)$.  Again, if $z'\notin
  R(x,z'')\cup R(z'',y)$, then $R(x,z')$, $R(z',y)$, $R(z'',y)$, $R(x,z'')$
  form a hyper-cycle, contradicting the assumption that $R$ is
  pyramidal. Therefore, $z'\in R(x,z'')\cup R(z'',y)$. Arguing as above, if
  $z'\in R(x,z'')\setminus R(z'',y)$, then from $y\in R(z'',y)\setminus
  R(x,z'')$ and \AX{(w3)}, we obtain $z''\in R(z',y)$, a
  contradiction. Similarly, $z'\in R(z'',y) \setminus R(x,z'')$ implies
  $z''\in R(x,z')$, a contradiction. Therefore, we have $z'\in R(x,z'')\cap
  R(z'',y)$ and thus $R(x,z')\cup R(z',y)\subset R(x,z'')\cup R(z'',y)$.
  
  Since $R(x,y)$ violates \AX{(u3)} we obtain $R(x,z'')\cup R(z'',y)\subset
  R(x,y)$, and thus there exists $z'''\in R(x,y)\setminus R(x,z'')\cup
  R(z'',y)$. Continuing this process yields an infinite sequence of
  distinct points $z,z',z'',z''',\dots \in R(x,y)$, which is impossible
  because $R(x,y)$ is finite. Therefore $R(x,y)$ cannot violate \AX{(u3)}.
\end{proof}
Example~\ref{ex:u3-py} shows that \AX{(u3)} does not imply that $R$ is
pyramidal:
\begin{example}
  \label{ex:u3-py}
  Let $R$ be defined on $V=\{a,b,c,d\}$ by $R(a,b)=\{a,b,c,d\}$,
  $R(a,c)=\{a,c\}$, $R(c,b)=\{c,b,d\}$, $R(a,d)=\{a,d\}$,
  $R(b,d)=\{b,c,d\}$, $R(c,d)=\{c,d\}$. Then $R$ is monotone and satisfies
  \AX{(u3)}, but $a,c,d$ violates \AX{(w)}. Thus, in particular, $R$ is
  not pyramidal.
\end{example}

The transit function $R$ in Example \ref{Ex:L2+wp-l1-w} satisfies \AX{(u3)},
but the sets $R(a,b)$, $R(a,c)$, and $R(b,c)$ violate \AX{(l1)}. The
transit function $R$ in Example \ref{ex:py+l1+u3-l2-ph} satisfies \AX{(u3)}
but violates \AX{(l2)}. That is, \AX{(u3)} $\nRightarrow$ \AX{(l1)} and
\AX{(u3)} $\nRightarrow$ \AX{(l2)}.  Fig.~\ref{fig:cex1}A shows that
\AX{(l1)} $\nRightarrow$ \AX{(u3)}. Therefore, Axiom \AX{(l1)} and
\AX{(u3)} are independent.
We conjecture that \AX{(l2)} implies \AX{(u3)} for monotone transit
functions since we were not successful in finding a counter-example.

The following Example \ref{w+wp+u3+hc-py} shows that the independent axioms
\AX{(w)}, \AX{(wp)}, and \AX{(u3)}, even together, do not imply that a
monotone transit $R$ function is pyramidal. 
\begin{example}\label{w+wp+u3+hc-py}
  Let $R$ be defined on $V=\{a,b,c,d,e,f\}$ by $R(a,b)=\{a,b\}$,
  $R(b,c)=R(b,d)=R(c,d)=\{b,c,d\}$, $R(d,e)=\{d,e\}$, $R(c,f)=\{c,f\}$, and
  the other sets are singletons or $V$. Here $R$ is monotone and satisfies
  \AX{(w)}, \AX{(wp)}, and \AX{(u3)} but is not pyramidal.
\end{example}

Let $<$ be a linear order on $V$ and consider a subset $V'\coloneqq
\{x_i|i=1,\dots , k\}$ such that $x_1<x_2<\dots <x_k$. Then the intervals
$[x_1,x_i]$ and $[x_1,x_j]$ do not overlap for any $x_i,x_j\in V'$.
Analogously, $[x_i,x_k]$ and $[x_j,x_k]$ do not overlap. This simple
observation suggests considering the following conditions:
\begin{description}
\item[\AX{(tb)}] For all $W\subseteq V$ with $|W|\ge 3$, there exists
  $x\in W$ such that $R(x,u)\subseteq R(x,v)$ or $R(x,v)\subseteq R(x,u)$
  for all $u,v\in W$.
\item[\AX{(hc)}] For all $W\subseteq V$ with $|W|\ge 3$, there exist
  distinct $x,y\in W$ such that $R(x,u)\subseteq R(x,v)$ or
  $R(x,v)\subseteq R(x,u)$ and $R(y,u)\subseteq R(y,v)$ or $R(y,v)\subseteq
  R(y,u)$ for all $u,v\in W$.
\end{description}
It follows immediately from the definition that \AX{(hc)} implies
\AX{(tb)}.  Example \ref{py-n3o} shows that \AX{(tb)} does not imply
\AX{(n3o)}. Fig.~\ref{fig:cex1}A shows that \AX{(n3o)} does not imply
\AX{(tb)}. The example in Fig.~\ref{fig:cex1}B, furthermore, shows that
\AX{(hc)} does not imply \AX{(wp)}. Moreover, the following Example
\ref{wp-hc} shows that \AX{(wp)} does not imply \AX{(hc)}. The transit
function corresponding to the set system in Fig.~\ref{fig:cex1}B trivially
satisfies \AX{(l2)} but violates \AX{(tb)}. The monotone transit function
in Example \ref{ex:py+l1+u3-l2-ph} is pyramidal and satisfies \AX{(tb)} but
violates \AX{(l2)}. In summary, \AX{(tb)} is independent of the axioms
\AX{(n3o)}, \AX{(wp)}, and \AX{(l2)}.
\begin{example}
  \label{wp-hc}
  Consider the transit function $R$ on $V=\{a,b,c,d,e,f\}$ defined by
  $R(a,b)=\{a,b\}$, $R(a,c)=R(a,f)=R(b,f)=R(c,f)=\{a,b,c,f\}$,
  $R(b,c)=\{b,c\}$, $R(b,d)=R(c,d)=\{b,c,d\}$, $R(b,e)=R(c,e)=\{b,c,d,e\}$,
  $R(d,e)=\{d,e\}$, $R(e,f)=\{e,f\}$, $R(a,d)=R(a,e)=R(d,f)=V$. Here $R$
  satisfies \AX{(wp)} but violates \AX{(hc)}.
\end{example}

\begin{lemma}
  \label{lem:hc<->tb}
  Let $R$ be a monotone transit function on $V$. $R$ satisfies \AX{(tb)} if
  and only if $R$ satisfies \AX{(hc)}.
\end{lemma}
\begin{proof}
  It follows immediately from the definition that \AX{(hc)} implies
  \AX{(tb)}. To prove the converse, we proceed by induction in
    $n=|V|$. First, consider the base case $n=3$ and assume that the
    monotone transit function $R$ on $V=\{a,b,c\}$ satisfies
    \AX{(tb)}. W.l.o.g., suppose $R(a,b)\subseteq R(a,c)$. Then by the
  monotonicity of $R$, we have $R(a,c)=\{a,b,c\}$ and thus 
  $R(b,c)\subseteq R(a,c)$, i.e., $R$ satisfies \AX{(hc)}.
    
  Now suppose $|V|>3$ and the assertion holds for all proper subsets of
  $V$. Since $R$ satisfied \AX{(tb)}, there exists $a\in V$ such that $R(a,u)$
  and $R(a,v)$ do not overlap for any pair $u$ and $v$. Consider
  $V'\coloneqq V\setminus\{a\}$. The induction hypothesis stipulates that
  there exist $b,c\in V'$ such that for all $u,v\in V'$, the sets $R(b,u)$
  and $R(b,v)$ do not overlap and the sets $R(c,u)$ and $R(c,v)$ do not
  overlap. Since $R$ satisfies \AX{(tb)} we know that $R(a,b)$ and $R(a,c)$
  do not overlap. W.l.o.g., we may assume $R(a,b)\subseteq R(a,c)$ and thus
  $b\in R(a,c)$. Monotonicity now implies $R(b,c)\subseteq R(a,c)$. If
  \AX{(hc)} does not hold, then there is $y\in V'$ such that
  $R(a,c)\overlaps R(c,y)$, which implies $R(b,c) \overlaps R(c,y)$ with
  $y\in V'$, a contradiction. Thus $R(y,c)\subseteq R(a,c)$ for all $y\in
  V'$. Since $R(a,a)=\{a\}\subseteq R(a,c)$, the statement holds for all
  $y\in V$, and thus $R$ satisfies \AX{(hc)}.
\end{proof}

In the following, we prove that \AX{(w)} is a generalization of \AX{(tb)}.
\begin{lemma}\label{tb->w}
  If $R$ is a monotone transit function satisfying \AX{(tb)} then
  it also satisfies \AX{(w)}.
\end{lemma}
\begin{proof}
  Since $R$ satisfies \AX{(tb)} for any set of three distinct vertices
  $V'=\{a,b,c\}$ there exists $x\in\{a,b,c\}$ such that $R(x,a')$ and
  $R(x,a'')$ do not overlap for $\{a',a''\}=V\setminus\{x\}$ and thus
  $R(x,a')\subseteq R(x,a'')$ or $R(x,a'')\subseteq R(x,a')$. This implies
  $a''\in R(x,a')$ or $a'\in R(x,a'')$, and thus $R$ satisfies \AX{(w)}.
\end{proof}
The converse is not true, however. The example in Fig.~\ref{fig:cex1}A
satisfies \AX{(w)} but violates \AX{(tb)} and thus also \AX{(hc)}. Next we
prove that the properties \AX{(hc)} and \AX{(tb)} are satisfied by every
pyramidal transit functions.
\begin{theorem}
	\label{thm:py->hc}
  Let $R$ be a pyramidal transit function, then $R$ satisfies \AX{(hc)}.
\end{theorem}
\begin{proof}
  Let $R$ be a pyramidal transit function on $V$. Then there exists  a linear
  order $<$ on $V$ such that $R(x,y)$ is an interval with respect to $<$
  for every $x,y\in V$. Let $V'\subseteq V$ and let $x=\min\{x'\in V'\}$,
  $y=\max\{x'\in V'\}$. Let $z,z'\in V'$ and $z\leq z'$.  Then $x\leq z \le
  z'\leq y$. Now, axioms \AX{(t1)} and \AX{(t2)} implies that $x,z'\in
  R(x,z')$. Therefore, $[x,z']\subseteq R(x,z')$, which implies $z\in
  R(x,z')$. Thus, $x\leq z \le z' \implies z\in R(x,z')$. Hence, \AX{(m)}
  implies $R(x,z)\subseteq R(x,z')$.  Analogously, $z\leq z'\leq y \implies
  z'\in [z,y]\subseteq R(z,y)$ and \AX{(m)} implies $R(z',y)\subseteq
  R(z,y)$.  In summary, $R$ satisfies \AX{(hc)}.
\end{proof}
The converse is not true, however. The Example~\ref{w+wp+u3+hc-py} gives a
monotone transit function satisfying \AX{(hc)} that is not pyramidal. Even
though it also satisfies \AX{(m)}, \AX{(hc)}, \AX{(w)}, \AX{(wp)} and
\AX{(u3)}, it is not pyramidal because $\mathscr{C}_R$ contains second
forbidden configuration in Eq.(\ref{eq:inthypg}).

In the following theorem, we prove that axiom \AX{(tb)} characterizes the
canonical transit function of a totally balanced clustering system:
\begin{lemma}
  \label{lem:hc-totbal}
  Let $R$ be a monotone transit function. Then $R$ satisfies \AX{(tb)} if
  and only if $(V,\mathscr{C}_R)$ is totally balanced.
\end{lemma}
\begin{proof}
  Let $R$ be a monotone transit function satisfying \AX{(tb)}.  Suppose
  that $(C_1, \dots, C_n)$ is a weak $\beta$-cycle in
  $(V,\mathscr{C}_R)$. Consider $x_i\in C_i\cap C_ {i+1}$ for
  $i=1,\dots,n-1$ and $x_n\in C_n\cap C_1$. Then, $R(x_i,x_{i+1})\subseteq
  C_{i+1}$ for $i=1,\dots,n-1$ and $R(x_1,x_n)\subseteq C_1$. Therefore,
  $(R(x_1,x_2), R(x_2,x_3), \dots ,R(x_{n-1},x_n))$ is a weak $\beta$-cycle
  in $(V,\mathscr{C}_R)$. These sets satisfy $R(x_1,x_2)\overlaps
  R(x_2,x_3) \overlaps \dots \overlaps R(x_{n-1},x_n) \overlaps
  R(x_n,x_1)\overlaps R(x_1,x_2)$. Consider the set
  $V'=\{x_1,\dots,x_n\}\subseteq V$. For each $x_i$, therefore, there
  exist $x_j,x_k\in V'$ such that $R(x_j,x_i)\overlaps R(x_i,x_k)$, thus
  the set $V'=\{x_1,\dots,x_n\}$ violates \AX{(tb)}, a contradiction.

  Conversely, suppose that $(V,\mathscr{C}_R)$ does not contain weak
  $\beta$-cycles. Assume, for contradiction that $R$ does not satisfies
  \AX{(tb)}, i.e., there exists $V'\subseteq V$ such that for any $x\in V'$
  there are $a,a'\in V'$ such that $R(x,a) \overlaps R(x,a')$. Let $x_1\in
  V'$. Then there are $x_2,x_3\in V'\setminus \{x_1\}$ such that
  $R(x_1,x_2)\overlaps R(x_1,x_3)$. Since $(V,\mathscr{C}_R)$ does not
  contain weak $\beta$-cycles, we have $x_1\in R(x_2,x_3)$. Therefore,
  $R(x_1,x_2)\cup R(x_1,x_3)\subseteq R(x_2,x_3)$.  Now, for $x_2\in V'$,
  there are distinct points $x_4,x_4'\in V'\setminus \{x_2\}$ such that
  $R(x_2,x_4)\overlaps R(x_2,x_4')$. Since $(V,\mathscr{C}_R)$ does not
  contain weak $\beta$-cycles, we have, $x_4\neq x_3\neq x_4'$. Note
    that only one of $x_4$ and $x_4'$ may coincide with $x_1$. Hence we may
    assume w.l.o.g.\ that $x_4\in V'\setminus
  \{x_1,x_2,x_3\}$. Furthermore, $R(x_1,x_2)\subseteq R(x_3,x_4)$, since
  otherwise $(V,\mathscr{C}_R)$ contains a weak $\beta$-cycle. Now, for
  $x_4\in V'$, there exist distinct $x_5,x_5'\in V'\setminus
  \{x_4\}$ such that $R(x_4,x_5)\overlaps R(x_4,x_5')$. Continuing in this
  manner, we step-by-step obtain an infinite sequence of distinct
    points $x_i \in V'$. This is impossible, however, since $V'$ is
  finite. Therefore, $R$ satisfies \AX{(tb)}.
\end{proof}
We note that Lemma~\ref{lem:hc-totbal} is similar to Prop.3 of
\cite{Brucker:10}, which states that a hypergraph H is totally balanced if
and only if it admits a so-called totally balanced ordering.

\begin{lemma}\label{tb->u3}
  If $R$ is a monotone transit function satisfying \AX{(tb)},
  then $R$ satisfies \AX{(u3)}.
\end{lemma}
\begin{proof}
  Lemma \ref{lem:hc-totbal} and Lemma \ref{tb->w} imply that $R$ satisfies
  \AX{(w)} and thus $\mathscr{C}_R$ cannot contain first forbidden
  configuration. Using the same arguments as in the proof of Lemma
  \ref{py->u3}, we conclude that $R$ satisfies \AX{(u3)}.
\end{proof}
Example~\ref{ex:u3-py} shows that the converse is not true.

The definition of weak $\beta$-cycles and the axiom \AX{(w)} suggest to
consider the following axiom for a transit function:
\begin{description}
\item[\AX{(tb')}] If $v_1,\dots,v_n\in V$ where $n\geq 3$ with
  $v_{k-1}\notin R(v_k,v_{k+1}), v_{k+1}\notin R(v_{k-1},v_k)$ for all $k$
  (indices taken modulo $n$), then there exists some $j$ with $1\leq j\leq
  n$ such that $v_j\in R(v_i,v_{i+1})$ for some $i\notin \{j,j-1\}$.
\end{description}
For $n=3$, \AX{(tb')} reduced to \AX{(w)}. A monotone transit function
satisfying \AX{(tb')} thus in particular satisfies \AX{(w)}.

\begin{lemma}
  \label{lem:tb'-totbal}
  Let $R$ be a monotone transit function. Then $R$ satisfies \AX{(tb')} if
  and only if $(V,\mathscr{C}_R)$ is totally balanced.
\end{lemma}
\begin{proof}
  Let $R$ be a monotone transit function satisfying \AX{(tb')}. Assume that
  $(V,\mathscr{C}_R)$ contains a weak $\beta$-cycle, $(C_1,\dots,C_n)$ for
  $n\geq 3$. Then there exists $x_i\in C_i\cap C_{i+1}$ for all $i$
  (indices taken modulo $n$) such that, no $x_j\in C_k$ for any $k\notin
  \{j,j+1\}$. By \AX{(m)}, we have, $R(x_i,x_{i+1})\subseteq C_{i+1}$ for
  all $i=1,\dots, n-1$ and $R(x_n,x_1)\subseteq C_1$, since $x_i,x_{i+1}\in
  C_{i+1}$ and $x_n,x_1\in C_1$. Therefore, no $x_j\in R(x_i,x_{i+1})$ for
  $i\notin \{j,j-1\}$.  Hence the points $x_1,\dots,x_n$ violates
  \AX{(tb')}, a contradiction. Therefore, $(V,\mathscr{C}_R)$ can
  not contain a weak $\beta$-cycle.
	
  Conversely, assume that $(V,\mathscr{C}_R)$ is totally balanced. Suppose
  $R$ violates \AX{(tb')}. Then there exist points $x_1,\dots ,x_n\in V$
  with $n\geq 3$ such that $x_{k-1}\notin R(x_k,x_{k+1}), x_{k+1}\notin
  R(x_{k-1},x_k)$ for all $k$ (indices taken modulo $n$), and no $x_j$
  holds $x_j\in R(x_i,x_{i+1})$ for some $i\notin \{j,j-1\}$. Then
  $(R(x_1,x_2),\dots ,R(x_{n-1},x_n), R(x_n,x_1))$ is a weak $\beta$-cycle,
  a contradiction.
\end{proof}

A useful property of totally balanced and monotone transit functions is the
following:
\begin{description}
\item[\AX{(tb2)}] For all $\emptyset\ne W\subseteq V$ there exist
  $x,y\in W$ such that $W\subseteq R(x,y)$.
\end{description}
\begin{lemma}
  Let $R$ be a monotone transit function. If $R$ satisfies \AX{(tb)}
  $R$ satisfies \AX{(tb2)}.
  \label{lem:tb2}
\end{lemma}
\begin{proof}
  If $W=\{x\}$ or $W=\{x,y\}$ then $W\subseteq R(x,y)$ by \AX{(t1)}. It
  therefore suffices to consider $|W|\ge 3$. First assume that $R$
  satisfies \AX{(tb)}. Thus there exists $x\in W$, such that, for all $u,v\in
  W$, either $R(x,u)\subseteq R(x,v)$ or $R(x,v)\subseteq R(x,u)$. Thus
  there is $w\in W$ such $R(x,u)\subseteq R(x,w)$, and thus $u\in R(x,w)$
  for all $u\in W$.  
\end{proof}
The converse is not true. The canonical transit function of the set
system in Fig.\ref{fig:cex1}A satisfies \AX{(tb2)} but violates \AX{(tb)}.

We summarize Lemmas~\ref{lem:hc<->tb}, \ref{lem:hc-totbal}, and
\ref{lem:tb'-totbal} in the following:
\begin{theorem}
  \label{thm:tb-equiv}
  Let $R$ be a monotone transit function on $V$. Then the following
  statements are equivalent:
  \begin{description}
  \item[(i)]   $R$ satisfies \AX{(tb)}
  \item[(ii)]   $R$ satisfies \AX{(hc)}
  \item[(iii)]   $R$ satisfies \AX{(tb')}
  \item[(iv)] $\mathscr{C}_R$ is totally balanced.
  \end{description}
\end{theorem}

\section{Second and third forbidden configurations}
\label{sect:23}

The second and third forbidden configurations in Eq.\ref{eq:inthypg} can be
understood as constraints on the mutual relationship of three sets $A$,
$B$, and $C$ that overlaps a fourth set $D$. We first consider set systems
$\mathscr{C}$ satisfying the following property:
\begin{description} 
\item[\AX{(P2)}] If $A,B,C,D\in \mathscr{C}$ satisfy $A\overlaps D$,
  $B\overlaps D$, and $C\overlaps D$, then $A\subseteq D\cup B\cup C$, or
  $B\subseteq D\cup A\cup C$, or $C\subseteq D\cup A\cup B$.
\end{description}
Then, $\mathscr{C}$ does not contain the second forbidden configuration in
Eq.\ref{eq:inthypg}. The converse is also true.

\begin{lemma}
  \label{lem:py->p2}
  Every pyramidal set system satisfies \AX{(P2)}.
\end{lemma}
\begin{proof}
  Consider $A,B,C,D\in\mathscr{C}$ such that $A\overlaps D$, $B\overlaps
  D$, and $C\overlaps D$. Since $R$ is pyramidal, $\mathscr{C}_R$ satisfies
  \AX{(WP)}.

  To this end, we first consider the case that one of the sets $A$, $B$,
  and $C$ is contained in one of the others. Then, trivially, at least one
  of the statements $A\subseteq D\cup B\cup C$, $B\subseteq D\cup A\cup C$,
  or $C\subseteq D\cup A\cup B$ is true. In the following, we, therefore,
  assume that none of the sets $A$, $B$, $C$ is contained in one of the
  others. First, consider the case that $A,B,C$ are pairwisely
  disjoint. Then there are points $x\in A\setminus D$, $y\in B\setminus D$
  and $z\in C\setminus D$. Thus $(A,D,C)$ is an $xz$-hyper-path not
  containing $y$, $(A,D,B)$ is an $xy$ hyper-path not containing $z$ and
  $(B,D,C)$ is an $yz$ hyper-path not containing $x$. This contradicts
  Duchets's characterization of interval hypergraphs and, thus, the
  assumption that $\mathscr{C}_R$ is pyramidal.

  Next, we assume that two sets overlap. W.l.o.g., suppose $A\overlaps
  B$. Obviously, we have either $D\nsubseteq A\cup B$ or $D\subset A\cup
  B$. In the first case, if $A\nsubseteq D\cup B$ and $B\nsubseteq A\cup
  D$, then $A,D,B$ violates \AX{(WP)}. Therefore, $A\subseteq D\cup B$ or
  $B\subseteq A\cup D$ must hold, leading to \AX{(P2)}.  In the second
  case, $D\overlaps C$ implies $\emptyset \ne D\cap C\subset A\cup
  B$. Therefore, $A,D,C$ violates \AX{(WP)} if $D\cap C\subset A$ and
  $B,D,C$ violates \AX{(WP)} if $D\cap C\subset B$; hence $D\cap C$ is not
  contained in $A$ or $B$ and thus $A,B,C$ violates \AX{(WP)} whenever
  $C\setminus (A\cup B)\neq \emptyset$. Thus we conclude that $C\subseteq
  A\cup B$ and thus also $C\subseteq D\cup A\cup B$. Hence the lemma.
\end{proof}

\begin{lemma}
  A set system $\mathscr{C}$ satisfying \AX{(L1)} also satisfies \AX{(P2)}.
\end{lemma}
\begin{proof}
  Let $A, B, C, D \in\mathscr{C}$ such that $A \overlaps D$, $B \overlaps
  D$, and $C\overlaps D$. The assertion of \AX{(P2)} is trivally true if
  one of the three sets $A$, $B$, and $C$ is contained in another ones. If
  this is not the case, \AX{(L1)} implies $A\cap B \subseteq D \subseteq
  A\cup B$, $A\cap C \subseteq D \subseteq A\cup C$, and $B\cap C \subseteq
  D \subseteq B\cup C$, and thus $(A \cap B) \cup (A\cap C) \cup (B\cap C)
  \subseteq D \subseteq (A \cup B) \cap (A\cup C) \cap (B\cup C)$. Using
  that intersection is distributive over union and union is distributive
  over intersection shows that $D=(A \cap B) \cup (A\cap C) \cup (B\cap
  C)$. Since $A\overlaps D$ in particular implies $D\ne\emptyset$, at least
  one of the three intersections is non-empty.  Assume, w.l.o.g., $A\cap
  B\ne\emptyset$. As noted above, the assertion of \AX{(P2)} holds for
  $A\subseteq B$ or $B\subseteq A$. It remains to consider the case
  $A\overlaps B$. This assumption, together with $A\overlaps D$ and
  \AX{(L1)}, implies $D\cap B \subseteq A \subseteq D\cup B \subseteq D\cup
  B\cup C$, i.e., the assertion of \AX{(P2)} also holds in this case.
\end{proof}
Axiom \AX{(P2)} can be translated into a transit axiom as follows:
\begin{description} 
\item[\AX{(p2)}] If $R(x,y)\overlaps R(s,t)$, $R(u,v)\overlaps R(s,t)$ and
  $R(p,q)\overlaps R(s,t)$ then $R(x,y)\subseteq R(u,v)\cup R(p,q)\cup
  R(s,t)$, or $R(u,v)\subseteq R(x,y)\cup R(p,q)\cup R(s,t)$, or
  $R(p,q)\subseteq R(x,y)\cup R(u,v)\cup R(s,t)$.
\end{description}

\begin{corollary}
  \label{cor:py->p2}
  Every pyramidal transit function satisfies \AX{(p2)}.
\end{corollary}
The converse is not true, however.  \AX{(p2)}, even together with \AX{(tb)}
and \AX{(wp)}, does not imply pyramidal:	
\begin{example}\label{ex:p2-p3}
  Let $R$ on $V=\{a,b,c,d,e\}$ be defined by $R(a,b)=\{a,b\}$,
  $R(b,c)=R(c,d)=\{b,c,d\}$, $R(b,d)=\{b,d\}$, $R(d,e)=\{d,e\}$,
  $R(a,d)=R(a,e)=R(b,e)=\{a,b,d,e\}$ and all other sets are singletons or
  $V$. $R$ satisfies \AX{(m)}, \AX{(wp)}, \AX{(tb)}, and \AX{(p2)}.  Here
  $R(a,b)\overlaps R(c,d)$, $R(d,e)\overlaps R(c,d)$, $R(a,d)\overlaps
  R(c,d)$.  That is, $\mathscr{C}_R$ contains a third forbidden configuration
  in Eq.\ref{eq:inthypg}. Therefore, $\mathscr{C}_R$ is not pyramidal. See
  Fig.\ref{fig:p234}A.
\end{example}

Example \ref{w+wp+u3+hc-py} shows that neither \AX{(w)} nor \AX{(wp)}
implies \AX{(p2)}. The transit function in Fig.~\ref{fig:cex1}C satisfies
\AX{(p2)} but violates \AX{(wp)}. In Example~\ref{wp-hc}, $R$ satisfies
\AX{(p2)}, but the points $c$, $e$, and $f$ violate \AX{(w)}.  Hence,
\AX{(w)}, \AX{(wp)}, and \AX{(p2)} are mutually independent. Also,
\AX{(wpy)} and \AX{(p2)} are independent. Moreover, the transit function
$R$ in Example~\ref{wp-hc} violates \AX{(tb)} as $\mathscr{C}_R$ contains
the pure-cycle $(R(c,f),R(c,e),R(e,f))$, and Example \ref{w+wp+u3+hc-py}
shows that \AX{(tb)} does not imply \AX{(p2)}; thus properties \AX{(tb)}
and \AX{(p2)} are independent.

Let us now turn to the third forbidden configuration: 
\begin{description}
\item[\AX{(P3)}] If $A\overlaps B$, $B\overlaps C$, $B\overlaps D$, and 
  $A\cup C\subseteq D$, then $A\cap C\neq \emptyset$.
 \item[\AX{(P3')}] If $A\overlaps B$, $B\overlaps C$, $A\cap
   C=\emptyset$, and $A\cup C\subseteq D$, then $B\subseteq D$.
\end{description}

Axioms \AX{(P3)} and \AX{(P3')} by design rule out the third forbidden
configuration in Eq.\ref{eq:inthypg}. The two formulations are equivalent.
To see this, consider four sets $A,B,C,D\subseteq V$ and assume $A\overlaps
B$, $B\overlaps C$, and $A\cup C\subseteq D$. Under these assumptions, we
have $B\cap D\ne\emptyset$ and thus $B\overlaps D$ is equivalent to
$B\not\subseteq D$. The statement ``$B\overlaps D$ implies $A\cap C\neq
\emptyset$'' thus is equivalent to the contra-positive of ``$A\cap
C=\emptyset$ implies $B\subseteq D$''.
\begin{lemma} 
  \label{lem:py->p3}
  Every pyramidal set system $\mathscr{C}$ satisfies \AX{(P3)}.
\end{lemma}
\begin{proof} 
  Let $\mathscr{C}$ be pyramidal. First, suppose \AX{(P3)} does not hold,
  i.e., there exist four sets $A$, $B$, $C$, $D$ that are intervals such
  that $A\overlaps B$, $B\overlaps C$, $B\overlaps D$ and $A\cup C\subseteq
  D$, but $A\cap C=\emptyset$. Then $B$ is located between $A$ and $C$,
  i.e., $A\cup B\cup C$ is an interval whose endpoints are located in
  $A\setminus B$ and $C\setminus B$, respectively, and therefore in
  $D$. Using again that $\mathscr{C}$ is pyramidal, we obtain $A\cup B\cup
  C\subseteq D$, contradicting $B\overlaps D$.
\end{proof}

\begin{figure}[t]
  \includegraphics[width=0.9\textwidth]{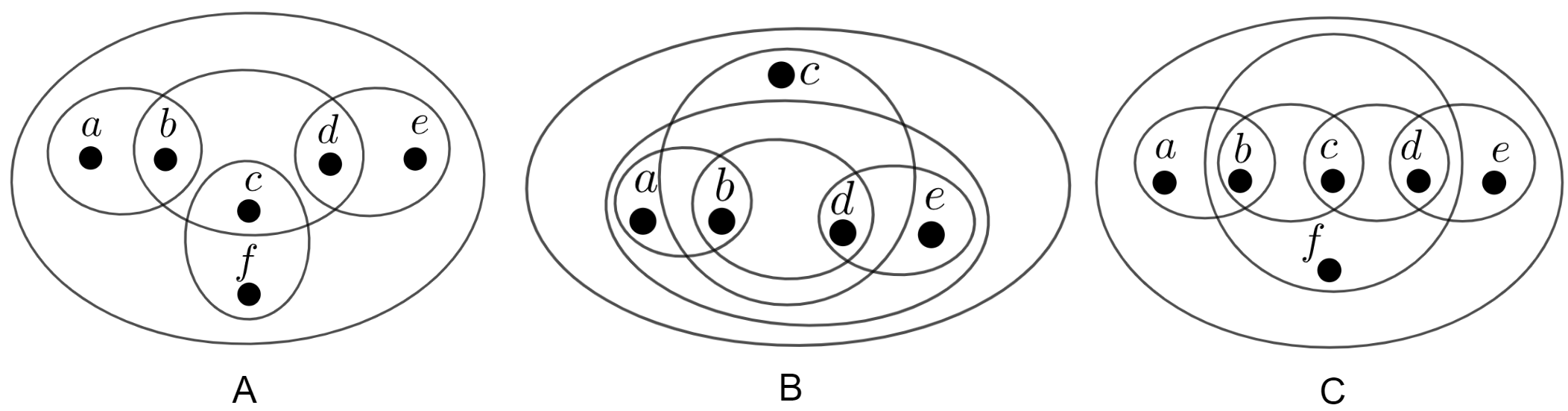}
  \caption{A, B, and C are set systems corresponding to the monotone
    transit functions in Examples~ \ref{w+wp+u3+hc-py} ,\; \ref{ex:p2-p3} \;
    and\; \ref{ex:p3-p4}\; respectively.}
  \label{fig:p234}
\end{figure}

\begin{lemma}
  If a set system $\mathscr{C}$ satisfies \AX{(L1)}, then it also satisfies
  \AX{(P3)}.
\end{lemma}
\begin{proof}
  Let $A, B, C, D\in \mathscr{C}$ such that $A\overlaps B$, $B\overlaps C$, $B\overlaps D$, and $A\cup C\subseteq D$. Since \AX{(L1)}
  holds, $A\overlaps B$ and $B\overlaps C$ together implies
  $A\subseteq C$, or $C\subseteq A$, or $A\cap C\subseteq B\subseteq A \cup
  C$.  The third alternative cannot occur because it yields $B\subseteq
  A\cup C\subseteq D$ and thus contradicts $B\overlaps D$. In the first two
  cases we have $A\cap C=A$ or $A\cap C=C$, and thus $A\cap C\ne\emptyset$,
  i.e., \AX{(P3)} holds.
\end{proof}

Like \AX{(P2)}, axiom \AX{(P3)} can be translated
directly to transit functions:
\begin{description} 
\item[\AX{(p3)}] If $R(x,y)\overlaps R(u,v)$, $R(x,y)\overlaps R(p,q)$,
  $R(x,y)\overlaps R(s,t)$ where $R(p,q)\subseteq R(u,v)$ and
  $R(s,t)\subseteq R(u,v)$, then $R(p,q)\cap R(s,t)\neq \emptyset$.
\end{description}
\begin{corollary}
	\label{cor:py->p3}
  Every pyramidal transit function satisfies \AX{(p3)}.
\end{corollary}
Example \ref{ex:p3-p4} shows that the converse is not true even if $R$
satisfies \AX{(tb)}, \AX{(wp)}, and \AX{(p3)}.
\begin{example}
  \label{ex:p3-p4}
  Let $R$ on $V=\{a,b,c,d,e,f\}$ be defined by $R(a,b)=\{a,b\}$,
  $R(b,c)=\{b,c\}$, $R(c,d)=\{c,d\}$,
  $R(b,d)=R(b,f)=R(c,f)=R(d,f)=\{b,c,d,f\}$, $R(d,e)=\{d,e\}$, and all
  other sets are either a singleton or $V$. $R$ is monotone, $R$ satisfies
  \AX{(tb)}, \AX{(wp)}, and \AX{(p3)}. However, $\mathscr{C}_R$ contains
  the fourth forbidden configuration Eq.\ref{eq:inthypg}, and thus $R$ is
  not pyramidal. See Fig. \ref{fig:p234}B.
\end{example}

From Example~\ref{ex:p2-p3}, we see that \AX{(p2)} does not imply
\AX{(p3)}. Example \ref{w+wp+u3+hc-py} (See Fig. \ref{fig:p234}C) shows
that \AX{(p3)} does not imply \AX{(p2)}. The transit function in
Fig~\ref{fig:cex1}A satisfies \AX{(p2)} and \AX{(p3)} but violates
\AX{(u3)}. $R$ in Example~\ref{w+wp+u3+hc-py} holds \AX{(u3)} and \AX{(p3)}
but violates \AX{(p2)}, and $R$ in Example~\ref{ex:p2-p3} holds \AX{(u3)}
and \AX{(p2)} but violates \AX{(p3)}. Therefore, \AX{(u3)} is independent
of \AX{(p2)} and \AX{(p3)}. Example~\ref{ex:p2-p3} and Fig.~\ref{fig:cex1}A
show that \AX{(p3)} and \AX{(tb)} are independent.  Moreover, \AX{(p3)}
implies neither \AX{(w)} by Fig.~\ref{fig:cex1}B nor \AX{(wp)} by
Fig.~\ref{fig:cex1}C. Also, \AX{(w)} and \AX{(wp)} need not imply
\AX{(p3)}. Therefore, axioms \AX{(tb)}, \AX{(p3)}, and \AX{(wp)} are
mutually independent. Similarly, \AX{(p2)}, \AX{(tb)}, \AX{(p3)} are
  mutually independent. Example~\ref{wp-hc} satisfies \AX{(w)}, \AX{(p2)},
\AX{(p3)}, \AX{(wp)} but violates \AX{(tb)}, and Fig.~\ref{fig:cex1}C
satisfies \AX{(w)}, \AX{(p2)}, \AX{(p3)}, \AX{(tb)} but violates \AX{(wp)}.
Moreover, the transit function in Fig~\ref{fig:cex1}C satisfies \AX{(p2)}
and \AX{(p3)} but violates \AX{(i)}. $R$ in Example~\ref{w+wp+u3+hc-py}
satisfies \AX{(i)} and \AX{(p3)} but violates \AX{(p2)}, and $R$ in
Example~\ref{ex:p2-p3} satisfies \AX{(i)} and \AX{(p2)} but violates
\AX{(p3)}. Therefore, \AX{(i)} is independent of \AX{(p2)} and \AX{(p3)}.
The following example, finally, shows that \AX{(p2)}, \AX{(p3)} ensure
neither \AX{(w)} nor \AX{(wp)}.
\begin{example}
  \label{ex:p2p3-wpy}
  Let $R$ on $V=\{a,b,c,d,e,f,g\}$ be defined by
  $R(a,b)=R(a,e)=R(b,e)=\{a,b,e\}$, $R(b,c)=R(b,f)=R(c,f)=\{b,c,f\}$,
  $R(b,d)=\{b,d\}$, $R(d,e)=\{d,e\}$, $R(a,c)=R(a,d)=R(d,c)=\{a,c,d\}$ and
  all other sets are singletons or $V$. $R$ satisfies \AX{(m)}, \AX{(p2)},
  and \AX{(p3)} but violates \AX{(w)} and \AX{(wp)}.
\end{example}

We finally note that \AX{(i)} is independent of \AX{(tb)}, \AX{(p2)}
  and \AX{(p3)}: The transit function in Fig~\ref{fig:cex1}C satisfies
\AX{(tb)}, \AX{(p2)} and \AX{(p3)} But violates \AX{(i)}. $R$ in
Example~\ref{w+wp+u3+hc-py} holds \AX{(i)} and \AX{(p3)} but violates
\AX{(p2)}, and $R$ in Example~\ref{ex:p2-p3} holds \AX{(i)} and \AX{(p2)}
but violates \AX{(p3)}. The transit function in Fig~\ref{fig:cex1}A
satisfies \AX{(i)} but violates \AX{(tb)}.

\section{Fourth and Fifth forbidden configurations}
\label{sect:45}

Let us now turn to the fourth and fifth forbidden configurations. We first
note that if $R$ satisfies \AX{(k)} in addition, then $\mathscr{C}_R$ is
closed. In this case, a fifth forbidden configuration in
Eq.\ref{eq:inthypg} also contains a fourth forbidden configuration as a
subhypergraph.  Therefore, it suffices to rule out the first four forbidden
configurations to ensure that a monotonous transit function satisfying
\AX{(k)} is pyramidal. This is in particularly the case for monotonous
transit functions satisfying \AX{(w)}. To address the fourth forbidden
configuration, we consider the following property:
\begin{description}
\item[\AX{(p4)}] If $u,v,y, x_1,\dots,x_n\in V$ for $n\ge 3$ satisfy
  \begin{itemize}
  \item[(i)] $x_i\notin R(x_{j},x_{j+1})$ for all $i$ and 
      $j\notin \{i-1,i\}$,
  \item[(ii)] $R(u,v)\cap R(x_1,x_2)\neq \emptyset$ and
      $R(u,v)\cap R(x_{n-1},x_n)\neq \emptyset$, and
  \item[(iii)]  $y\in R(u,v)\setminus R(x_i,x_{i+1})$ for all $i$,
  \end{itemize}
  then $R(x_1,x_2)\subseteq R(u,v)$ or $R(x_{n-1},x_n)\subseteq R(u,v)$.
\end{description}
It is not difficult to check that the system of transit sets
$\mathscr{C}_R$ of a monotone transit function $R$ satisfying \AX{(p4)} can
not contain the fourth forbidden configuration in Eq.\ref{eq:inthypg} for
any number $n$ of ``small sets''. The converse is also true.

\begin{lemma}
  If a monotone transit function satisfies \AX{(l1)}, then it also
  satisfies \AX{(p4)}.
\end{lemma}
\begin{proof}
  Suppose there exist $u,v,y, x_1,\dots,x_n\in V$ for $n\ge 3$ satisfying
  the pre-conditions (i), (ii), and (iii) of axiom \AX{(p4)} such that
  $R(x_1,x_2)\nsubseteq R(u,v)$ and $R(x_{n-1},x_n)\nsubseteq R(u,v)$,
  i.e.,  \AX{(p4)} is violated. Then,
  $R(x_1,x_2)\overlaps R(u,v)\overlaps R(x_{n-1},x_n)$, but
  $R(u,v)\nsubseteq R(x_1,x_2)\cup R(x_{n-1},x_n)$, i.e., condition
  \AX{(l1)} is violated as well.
\end{proof}
\begin{lemma}
  \label{lem:py->p4}
  Every pyramidal transit function satisfies \AX{(p4)}.
\end{lemma}
\begin{proof}
  Let $R$ be a pyramidal transit function violating \AX{(p4)}. Then there
  exist points $u,v,y\in V$ and $x_1,\dots,x_n\in V$ for $n\ge 3$ such that
  $R(u,v)\cap R(x_1,x_2)\neq \emptyset$,
  $R(u,v)\cap R(x_{n-1},x_n)\neq \emptyset$, $x_i\notin R(x_{j},x_{j+1})$
  for $j\notin \{i-1,i\}$, and $y\in R(u,v)\setminus R(x_i,x_{i+1})$ for
  all $i$, but $R(x_1,x_2)\nsubseteq R(u,v)$ and
  $R(x_{n-1},x_n)\nsubseteq R(u,v)$.  Since $\mathscr{C}_R$ is pyramidal,
  there is an order $<$ on $V$ with respect to which all transit sets are
  intervals. Thus, both
  $I_1\coloneqq R(x_1,x_2)\cup \dots \cup R(x_{n-1},x_n)$ and
  $I_2\coloneqq R(x_1,x_2)\cup R(u,v)\cup R(x_{n-1},x_n)$ are unions of
  non-disjoint intervals with respect to $<$ and thus again
  intervals. Since $y\in R(u,v)\setminus R(x_i,x_{i+1})$ for all $i$, we
  have $y\in I_2\setminus I_1$. Since, $R(x_1,x_2)\nsubseteq R(u,v)$ and
  $R(x_{n-1},x_n)\nsubseteq R(u,v)$ by assumption, $R(u,v)$ is an interval
  delimited by some $z_1\in R(x_1,x_2)$ and $z_n\in
  R(x_{n-1},x_n)$. Therefore, $I_2\subseteq I_1$, a contradiction to
  $y\in I_2\setminus I_1$.
\end{proof}
       
The transit function in Example~\ref{ex:p3-p4} satisfies \AX{(m)},
\AX{(tb)}, \AX{(p2)}, and \AX{(p3)} but violates \AX{(p4)}. The transit
function in Example~\ref{ex:p2-p3} satisfies \AX{(m)}, \AX{(tb)},
\AX{(p2)}, and \AX{(p4)} but violates \AX{(p3)}. The transit function in
Example~\ref{w+wp+u3+hc-py} satisfies \AX{(m)}, \AX{(tb)}, \AX{(p4)}, and
\AX{(p3)} but violates \AX{(p2)}. The canonical transit function of the set
system in Fig.~\ref{fig:cex1}A satisfies \AX{(m)}, \AX{(p2)}, \AX{(p4)},
and \AX{(p3)} but violates \AX{(tb)}. In summary, therefore, the axioms
\AX{(p2)}, \AX{(p3)}, \AX{(p4)}, and \AX{(tb)} are mutually independent for
monotone transit functions.

\begin{lemma}
  \label{lem:p4->wp}
  If $R$ satisfies \AX{(p4)}, then $R$ satisfies \AX{(wp)}.
\end{lemma}
\begin{proof}
  Let $A,B,C\in \mathscr{C}_R$ with pairwise non-empty
  intersections. Suppose that $A,B,C$ violates \AX{(wp)}.  \emph{Case 1 :
  $A\cap B\cap C=\emptyset$.} Let $x_1\in A\cap B$, $x_2\in B \cap C$ and
  $x_3\in C\cap A$. Since $A\cap B\cap C=\emptyset$, we have, $x_1\notin
  R(x_2,x_3)$ and $x_3\notin R(x_1,x_2)$. Since $A,B,C$ violates \AX{(wp)},
  there exists $y\in A$ such that $y\notin B\cup C$. Also $A\cap
  R(x_2,x_3)=\emptyset$ and $A\cap R(x_1,x_2)=\emptyset$. Therefore,
  $x_1,x_2,x_3$ and $A$ violates \AX{(p4)}.  \emph{Case 2 : $A\cap B\cap
  C\neq\emptyset$.} Let $x_2\in A \cap B\cap C$. Also, there exists $x_1\in
  A\setminus (B\cup C)$, $x_3\in B\setminus (A\cup C)$ and $x_4\in
  C\setminus (A\cup B)$. Then the points $x_1,x_2,x_3$ and the set
  $R(x_2,x_4)$ violates \AX{(p4)}. Therefore, $R$ satisfying \AX{(p4)} also
  satisfies \AX{(wp)}.
\end{proof}

If $\mathscr{C}_R$ contains a weak $\beta$-cycle $\{C_1,\dots, C_n\}$ such
that $\mathcal{C}_i\setminus (\mathcal{C}_{i+1} \cup \mathcal{C}_{i-1})
\neq\emptyset$ for some $1\leq i\leq n$, then, $R$ violates \AX{(p4)}. That
is, if $R$ is monotone and satisfies \AX{(p4)}, then every weak
$\beta$-cycle $\{C_1,\dots,C_n\}\subseteq \mathscr{C}_R$ with $n\geq 3$,
satisfies $\mathcal{C}_i\setminus
(\mathcal{C}_{i+1}\cup\mathcal{C}_{i-1})=\emptyset$ for all $1\le i\le n$,
where indices are taken modulo $n$).

We note that \AX{(p4)} does not imply \AX{(w)}. For example, Let
$\mathscr{C}=\{\{a,b\}, \{b,c\}, \{c,a\}, V\}$ be a set system on
$V=\{a,b,c,d\}$. Here, the canonical transit function satisfies \AX{(p4)}
but violates \AX{(w)}. Together, axioms \AX{(tb)} and \AX{(p4)} provide a
strengthening of weak pyramidality:
\begin{corollary}
  \label{cor:tbp4->wpy}
  Let $R$ be a monotone transit function satisfying \AX{(tb)} and
  \AX{(p4)}, then $\mathscr{C}_R$ is weakly pyramidal.
\end{corollary}
\begin{proof}
  For a monotone transit function, \AX{(tb)} implies \AX{(w)} from
  Lemma~\ref{tb->w}, and \AX{(p4)} implies \AX{(wp)} from
  Lemma~\ref{lem:p4->wp}. Therefore, $\mathscr{C}_R$ is weakly pyramidal.
\end{proof}
Example~\ref{ex:wpy-tbp234} shows that the converse of
Cor.~\ref{cor:tbp4->wpy} is not true.
\begin{example}
  \label{ex:wpy-tbp234}
  Let $R$ on $V=\{a,b,c,d,e,f\}$ be defined by $R(a,b)=\{a,b\}$,
  $R(b,c)=\{b,c\}$, $R(c,d)=\{c,d\}$, $R(a,d)=\{a,d\}$,
  $R(b,e)=R(c,e)=\{b,c,e\}$, $R(e,f)=\{e,f\}$ and all other sets are
  singletons or $V$. $R$ is weakly pyramidal but violates \AX{(tb)},
  \AX{(p2)}, \AX{(p3)}, and \AX{(p4)}.
\end{example}

Since \AX{(tb)} implies \AX{(w)} and \AX{(p4)} implies \AX{(wp)}, we see
that \AX{(w)} and \AX{(p4)} imply \AX{(wpy)}. However,
Examples~\ref{ex:wpy-tbp234}, \ref{ex:p2-p3}, \ref{ex:p2p3-wpy}, and
\ref{w+wp+u3+hc-py} show that \AX{(wpy)} is independent of each of the
conditions \AX{(p2)}, \AX{(p3)}, \AX{(p4)}, and \AX{(tb)}. The example
  in Fig\ref{fig:cex1}A shows that \AX{(w)} and \AX{(p4)} together do not
  imply \AX{(tb)}.
  
We are now in the position to give a characterization of pyramidal
  transit functions as a translation of Tucker's characterization of
  interval hypergraph \cite{Tucker:72,Trotter:76,Duchet:84} to the realm of
  transit functions.
\begin{theorem}
  \label{thm:final}
  A transit function $R$ satisfies \AX{(m)}, \AX{(tb)}, \AX{(p2)},
  \AX{(p3)}, and \AX{(p4)} if and only if $R$ is pyramidal.
\end{theorem}

\begin{proof}
  Let $R$ be a monotone transit function with the set of transit sets
  $\mathscr{C}_R$.  If $R$ is pyramidal, then it is, in particular, totally
  balanced and thus satisfies \AX{(tb)} by Thm.~\ref{thm:tb-equiv}.
  Furthermore, $R$ satisfies axioms \AX{(p2)}, \AX{(p3)}, and \AX{(p4)} by
  Corollaries \ref{cor:py->p2}, \ref{cor:py->p3}, and Lemma
  \ref{lem:py->p4}, respectively. For the converse, we assume that
  \AX{(tb)}, \AX{(p2)}, \AX{(p3)}, and \AX{(p4)} hold. We argue that this
  implies that $\mathscr{C}_R$ cannot contain any of the five forbidden
  configurations of Tucker's characterization of interval hypergraphs.
  First, we observe that \AX{(tb)} implies that $\mathscr{C}_R$ is totally
  balanced and thus does not contain the first forbidden configuration,
  i.e., weak $\beta$-cycles. Axioms \AX{(p2)} and \AX{(p3)} imply that
  $\mathscr{C}_R$ satisfies \AX{(P2)} and \AX{(P3)}, which by construction
  exclude the second and third forbidden configuration,
  respectively. Finally, if $R$ satisfies \AX{(p4)}, then $\mathscr{C}_R$
  does not contain the fourth forbidden configuration. We have already
  noted above that, since \AX{(tb)} implies \AX{(k)}, every fifth forbidden
  configuration contains a fourth forbidden configuration. Thus \AX{(tb)}
  and \AX{(p4)} imply that a fifth forbidden configuration cannot appear in
  $\mathscr{C}_R$. Taken together, $\mathscr{C}_R$ is pyramidal.
\end{proof}

The examples discussed above, furthermore, show that \AX{(p2)}, \AX{(p3)},
\AX{(p4)}, and \AX{(tb)} remain independent of each other for monotone
transit functions and thus no subset of these four conditions is
sufficient.
  
\section{Discussion}
\label{discsn}

In this contribution, we have obtained a characterization of the pyramidal
transit function based on Tucker's characterization \cite{Tucker:72} of
interval hypergraphs in terms of forbidden
configurations. Theorem~\ref{thm:final} used two first-order axioms
\AX{(p2)} and \AX{(p3)} as well as the monadic second-order axioms
\AX{(tb)} and \AX{(p4)}. All axioms presented in this paper except
\AX{(p4)}, \AX{(tb2)} and the three equivalent conditions \AX{(hc)}, \AX{(tb)}, and
\AX{(tb')}, are first order axioms. We have proved, therefore, that the transit
functions of union closed set systems and weakly pyramidal set systems are
first order axiomatizable. Also, paired hierarchies are proved to be first
order axiomatizable \cite{Bertrand:17}. In the light of
  Theorem~\ref{thm:final}, we strongly suspect that pyramidal clustering
  systems will not be first-order axiomatizable. Thus we may note that
  \AX{(py)} lies between two pairs of first order axiomatizable clustering
  systems, namely \AX{(PH)} and \AX{(WPY)}, and \AX{(UC)} and \AX{(WPY)},
  respectively.

Moreover, we resolved an open question from earlier work \cite{Changat:21w}
by showing that a transit function is union-closed if and only if it
satisfies \AX{(u)} and \AX{(w)}. We introduced a pair of conditions
\AX{(l1)} and \AX{(l2)} that is sufficient for pyramidal transit functions
and elaborated on necessary conditions, particularly the weakly pyramidal
property and total-balancedness. The implications among all the axioms
discussed in this contribution are summarized in Fig.~\ref{fig:summary}.

\begin{figure}
  \begin{center}
    \includegraphics[width=0.9\textwidth]{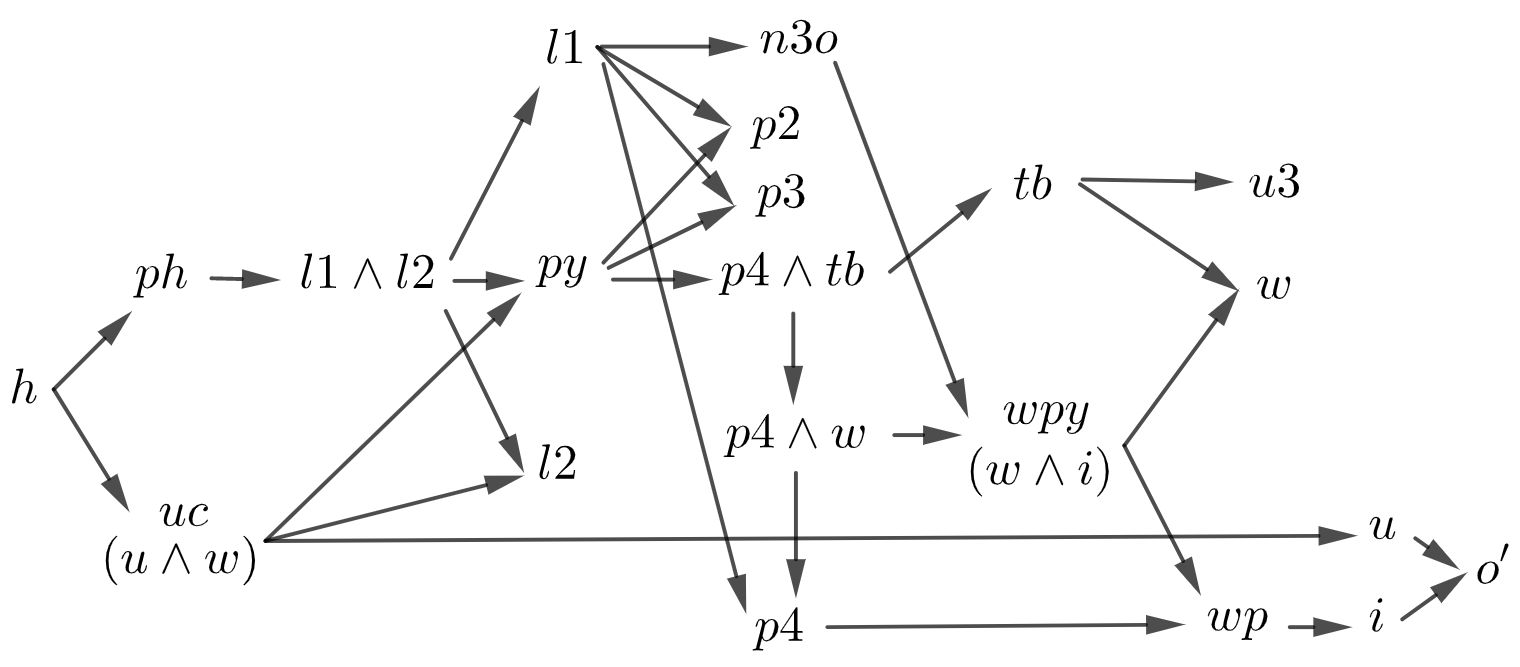}
  \end{center}
  \caption{Summary of implications among properties of monotone transit
    functions.}
  \label{fig:summary}
\end{figure}

Some relationships remain unclear. In particular, it remains open whether
\AX{(l2)} and \AX{(w)} together are sufficient to imply \AX{(tb)} or
\AX{(u3)}. In some cases, furthermore, conditions that seem natural for
transit functions do not have obvious ``translations'' to the more general
setting of set systems. For example, we do not have an equivalent for
\AX{(l2)} or \AX{(p4)} in the language of set systems. One might argue
  that \AX{(p4)} looks rather contrived beyond being designed to rule out
  the fourth forbidden configuration. It would certainly be interesting to
  know whether it could be replaced by simpler condition with a more direct
  translation to set systems.

A hypergraph is \emph{arboreal} if there is a tree $T$ such that every
hyperedge induces a connected subgraph, i.e., a subtree, of $T$
\cite[ch.5.4]{Berge:89}. An interval hypergraph thus is an arboreal
hypergraph for which $T$ is a path. Arboreal hypergraphs have been studied
from the point of view of cluster analysis and their corresponding
dissimilarities in \cite{Brucker:05}, suggesting that the transit functions
associated with binary arboreal clustering system also will be of interest.

\subsection*{Acknowledgements}
This research work was performed in part while MC was visiting the Max
Plank Institute for Mathematics in the Sciences (MPI-MIS) in Leipzig and
Leipzig University's Interdisciplinary Center for Bioinformatics (IZBI). MC
acknowledges the financial support of the MPI-MIS, the hospitality of the
IZBI, and the Commission for Developing Countries of the International
Mathematical Union (CDC-IMU) for providing the individual travel fellowship
supporting the research visit to Leipzig. This work was supported in part
by SERB-DST, Ministry of Science and Technology, Govt.\ of India, under the
MATRICS scheme for the research grant titled ``Axiomatics of Betweenness in
Discrete Structures'' (File: MTR/2017/000238). The work of AVS is supported
by CSIR, Government of India (File: 09/0102(12336)/2021-EMR-I).


\end{document}